\documentclass[reqno,10pt]{amsart}

\usepackage{amsmath}
\usepackage{amsfonts}
\usepackage{amssymb}
\usepackage{amsxtra}
\usepackage{amstext}
\usepackage{latexsym}
\usepackage{amsthm}
\usepackage{graphicx}
\usepackage[all,2cell]{xy}
\usepackage{dsfont} 
\usepackage{mathrsfs}
\usepackage{stmaryrd} 

\newtheorem{theorem}{Theorem}[section]
\newtheorem{lemma}[theorem]{Lemma}
\newtheorem{proposition}[theorem]{Proposition}

\theoremstyle{definition}
\newtheorem{definition}[theorem]{Definition}

\newtheorem{def-thm}[theorem]{Definition-Theorem}
\newtheorem{def-lemma}[theorem]{Definition-Lemma}
\newtheorem{def-prop}[theorem]{Definition-Proposition}

\theoremstyle{remark}
\newtheorem{remark}[theorem]{Remark}

\numberwithin{theorem}{section}
\numberwithin{equation}{section}
\numberwithin{figure}{section}
\numberwithin{table}{section}

\DeclareMathOperator{\id}{id}

\DeclareMathOperator{\ob}{ob}

\DeclareMathOperator{\op}{op}

\DeclareMathOperator{\sSet}{\mathbf{sSet}}

\DeclareMathOperator{\ub}{\underline{\bullet}}
\DeclareMathOperator{\usSet}{\underline{\mathbf{sSet}}}

\newdir{ >}{{}*!/-10pt/@{>}}

\usepackage{eucal}

\newcommand{\B}{\mathcal{B}}

\newcommand{\A}{\mathcal{A}}

\newcommand{\V}{\mathcal{V}}
\renewcommand{\S}{\mathcal{S}}

\newcommand{\uV}{\protect\underline{\mathcal{V}}}

\newcommand{\uS}{\protect\underline{\mathcal{S}}}
\renewcommand{\i}{\textsc{i}}
\newcommand{\uA}{\protect\underline{\mathcal{A}}}

\newcommand{\uB}{\protect\underline{\mathcal{B}}}

\begin{document}

\title{Cylinder, Tensor and Tensor-Closed Module}

\author{Seunghun Lee}

\address{Department of Mathematics, Konkuk University,
Kwangjin-Gu Hwayang-dong 1, Seoul 143-701, Korea}

\email{mbrs@konkuk.ac.kr}

\thanks{This research was supported by Basic Science Research
Program through the National Research Foundation of Korea(NRF)
funded by the Ministry of Education, Science and
Technology(2009-0076403)}


\date{}


\begin{abstract}
The purpose of this note is to show that, if $\V$ is a closed monoidal category, the following three notions are equivalent.
\begin{enumerate}
  \item Category with $\V$-structure and cylinder.
  \item Tensored $\V$-category.
  \item Tensor-closed $\V$-module.
\end{enumerate}
As an application we will show that, if $\V$ is closed and symmetric, then  given a category $\S$ there is an one-to-one correspondence between the set of $\V$-structures with cylinder and path on $\S$ introduced by Quillen and the set of closed $\V$-module structures on $\S$ introduced by Hovey.
\end{abstract}

\maketitle

\setcounter{tocdepth}{1}
\tableofcontents


\section{Introduction}

Let $\V=(\V,\otimes,I,a,l,r)$ be a monoidal category.
The purpose of this note is to show that, if $\V$ is closed, the following three notions are equivalent.
Precise statements will be given in Theorem \ref{thm:cp2.7.5} and Theorem \ref{thm:cp2.6.9}.
We would like to emphasize that $\V$ need not be symmetric.
\begin{enumerate}
  \item\label{itm:cp2.1.1} Category with $\V$-structure and cylinder.
  \item\label{itm:cp2.1.2} Tensored $\V$-category.
  \item Tensor-closed $\V$-module.
\end{enumerate}
As an application we will show in Theorem \ref{thm:cp2.1.11} that, if $\V$ is closed and symmetric, then  given a category $\S$ there is an one-to-one correspondence between the set of $\V$-structures with cylinder and path on $\S$ introduced by Quillen and the set of closed $\V$-module structures on $\S$ introduced by Hovey.

Our reference for the $\V$-category, or the category enriched over $\V$,  is \cite{kelly-05}.
A $\V$-category has an underlying ordinary category. A category $\S$ with $\V$-structure defined shortly
has a canonically associated $\V$-category whose underlying category is isomorphic to $\S$ as ordinary categories.
We refer to section \ref{sec:cp2.v-st} for the detail.
There is no essential difference between the $\V$-category and the category with $\V$-structure. But the category with $\V$-structure is the right one for us to carry out the comparison with.

In \cite{quillen-67}, Quillen introduced the simplicial category in the course of defining the simplicial model category. A simplicial category is a category with $\sSet$-structure where $\sSet$ is the category of the simplicial sets.

\begin{definition}[cf.~Definition 1 in chapter 2.1 of \cite{quillen-67}]
\label{def:cp2.4.5}
  Let $\V$ be a monoidal category.
  We say that a category $\S$ has a $\V$-structure if
  \begin{enumerate}
    \item there exists a functor
        \begin{equation}
        \label{eqn:cp2.4.17}
          \uS(-,-):\S^{\op}\times\S\rightarrow \V,
        \end{equation}
    \item for every $X,Y,Z\in\ob\S$, there exists a morphism
        \begin{equation}
        \label{eqn:cp2.4.18}
          \ub_{X,Y,Z}:\uS(Y,Z)\otimes\uS(X,Y)\rightarrow\uS(X,Z)
        \end{equation}
        of $\V$, and
    \item there exists a natural isomorphism
        \begin{equation}
        \label{eqn:cp2.4.15}
          \varphi:\S(-,-)\xrightarrow{\cong}\V( I ,\uS(-,-))
        \end{equation}
  \end{enumerate}
  such that
  \begin{enumerate}
  \setcounter{enumi}{3}
    \item the following diagram commutes
        \begin{equation}
        \label{eqn:cp2.4.2}
          \xymatrix@C=-1em{
            (\uS(Z,W)\otimes\uS(Y,Z))\otimes\uS(X,Y) \ar[rr]^{a}
              \ar[d]_{\ub\otimes1} &
              &\uS(Z,W)\otimes(\uS(Y,Z)\otimes\uS(X,Y))
              \ar[d]^{1\otimes\ub}\\
            \uS(Y,W)\otimes\uS(X,Y) \ar[r]_(0.63){\ub} & \uS(X,W)
              & \uS(Z,W)\otimes\uS(X,Z), \ar[l]^(0.63){\ub}}
        \end{equation}
    \item for every $X,Y,Z\in\ob\S$ and $f\in\S(X,Y)$, the morphism
        \begin{equation}
        \label{eqn:cp2.4.21}
          \uS(f,Z):\uS(Y,Z)\rightarrow\uS(X,Z)
        \end{equation}
        of $\V$ is the composition
        \begin{equation}
        \label{eqn:cp2.4.22}
          \uS(Y,Z)\xrightarrow{r^{-1}} \uS(Y,Z)\otimes I \xrightarrow{1\otimes\varphi(f)} \uS(Y,Z)\otimes\uS(X,Y) \xrightarrow{\ub} \uS(X,Z)
        \end{equation}
        and
    \item for every $X,Y,Z\in\ob\S$ and $g\in\S(Y,Z)$, the morphism
        \begin{equation}
        \label{eqn:cp2.4.23}
          \uS(X,g):\uS(X,Y)\rightarrow\uS(X,Z)
        \end{equation}
        of $\V$ is the composition
        \begin{equation}
        \label{eqn:cp2.4.24}
          \uS(X,Y)\xrightarrow{l^{-1}} I\otimes\uS(X,Y) \xrightarrow{\varphi(g)\otimes 1} \uS(Y,Z)\otimes\uS(X,Y) \xrightarrow{\ub} \uS(X,Z).
        \end{equation}
  \end{enumerate}
\end{definition}

Besides the model structure, a simplicial model category has two additional structures, cylinder and path.
To define cylinder \emph{and} path, we need $\V$ to be not only closed but also symmetric. The symmetry of $\sSet$ was used to define $\sharp(\psi)$ in Definition 3 in chapter 2.1 of \cite{quillen-67}.
Below we recall the definition of cylinder. The definition of path will be given in Definition \ref{def:cp2.1.28}.

\begin{definition}[cf.~Definition 3 in chapter 2.1 of \cite{quillen-67}]
\label{def:cp2.4.1}
  Let $\V$ be a closed monoidal category.
  Let $\S$ be a category with $\V$-structure.
  We say that $\S$ has cylinder if
  \begin{enumerate}
    \item for every $K\in\ob\V$ and $X\in\ob\S$, there exists an object
        \begin{equation}
        \label{eqn:cp2.5.3}
          K\otimes X
        \end{equation}
        of $\S$,
    \item for every $K\in\ob\V$ and $X\in\ob\S$, there exists a morphism
        \begin{equation}
        \label{eqn:cp2.4.0}
          \alpha_{K,X}:K\rightarrow \uS(X,K\otimes X)
        \end{equation}
        of $\V$, and
    \item for every $K\in\ob\V$ and $X,Y\in \ob\S$, there exists an isomorphism
        \begin{equation}
        \label{eqn:cp2.4.1}
          \underline\phi_{K,X,Y}:\uS(K\otimes X,Y)
            \rightarrow \uV(K,\uS(X,Y))
        \end{equation}
        of $\V$ such that the following diagram commutes
        \begin{equation}
        \label{eqn:cp.1.1}
          \xymatrix{
            \uS(K\otimes X,Y)\otimes K
              \ar[d]_{\underline\phi_{K,X,Y}\otimes 1}
              \ar[rr]^(0.44)
                {1\otimes\alpha_{K,X}}
              && \uS(K\otimes X , Y)\otimes \uS(X,K\otimes X)
              \ar[d]^{\ub}\\
            \uV(K, \uS(X,Y))\otimes K
              \ar[rr]_(0.58){\epsilon_{\uS(X,Y)}}
                && \uS(X,Y)}
        \end{equation}
        where $\epsilon_{\uS(X,Y)}$ is \eqref{eqn:elno.5.3}.
  \end{enumerate}
  We call the above triple $(K\otimes X,\alpha_{K,X},\{\underline\phi_{K,X,Y}\}_{Y\in\ob\S})$ a $K$-cylinder of $X$.
\end{definition}

\begin{remark}
\label{rem:cp2.1.4}
  If a monoidal category $\V$ is closed, then equipped with the $\V$-structure in Proposition \ref{pro:cp2.3.5}, $\V$ has cylinder. See Proposition \ref{pro:cp2.7.1} for the precise statement.
\end{remark}

\begin{remark}
  By an abuse of notion, we will write $\alpha$ and $\underline\phi$ instead of $\alpha_{K,X}$ and $\underline\phi_{K,X,Y}$ respectively when no confusion is likely to occur.
\end{remark}

Given $K\in\ob\V$ and $X\in\ob\S$, the pair $(K\otimes X, \alpha_{K,X})$ is essentially unique.

\begin{proposition}
\label{pro:esht.1.1}
  Let $\V$ be a closed monoidal category.
  Let $\S$ be a category with $\V$-structure.
  Let $K\in\ob\V$ and $X\in\ob\S$.
  Let
  $(K\otimes X,\alpha_{K,X},\{\underline\phi_{K,X,Y}\}_{Y\in\S})$
  and $(K\otimes' X,\alpha_{K,X}',\{\underline\phi_{K,X,Y}'\}_{Y\in\S})$
  be two $K$-cylinders of $X$.
  Then
  there exists an unique isomorphism
  \[f:K\otimes X\rightarrow K\otimes' X\] of $\S$
  such that
  \[
    \alpha_{K,X}'=\uS(X,f)\cdot\alpha_{K,X}
  \]
  holds.
\end{proposition}

The functoriality and the naturality are hidden in the definition of cylinder.

\begin{proposition}
\label{pro:esht.1.2}
  Let $\V$ be a closed monoidal category.
  Let $\S$ be a category with $\V$-structure.
  We assume that, for every $K\in\ob\V$ and $X\in\ob\S$,
  a $K$-cylinder of $X$  exists and we fix one
  $(K\otimes X,\alpha_{K,X},\{\underline\phi_{K,X,Y}\}_{Y\in\S})$. Then
  \begin{enumerate}
    \item there exists an unique bifunctor
        \begin{equation}
        \label{eqn:cp2.1.13}
          -\otimes-:\V\times\S\rightarrow \S
        \end{equation}
        such that
        \begin{enumerate}
          \item on objects, it maps $(K,X)$ to $K\otimes X$ and
          \item
              if $L\in\sSet$,
              $Y\in\S$, $u\in\sSet(K,L)$ and $v\in\S(X,Y)$, then
              \begin{equation}
                \xymatrix{
                  K \ar[r]^(0.3){\alpha_{K,X}} \ar[d]_u
                    & \uS(X,K\otimes X) \ar[d]^{\uS(X,u\otimes X)}\\
                  L \ar[r]_(0.3){\alpha_{L,X}} & \uS(X,L\otimes X)}
              \end{equation}
              and
              \begin{equation}
                \xymatrix{
                  K \ar[rr]^(0.4){\alpha_{K,X}}
                    \ar[d]_{\alpha_{K,Y}}
                    && \uS(X,K\otimes X)
                    \ar[d]^{\uS(X,K\otimes v)}\\
                  \uS(Y,K\otimes Y) \ar[rr]_{\uS(v,K\otimes Y)}
                    && \uS(X, K\otimes Y)}
              \end{equation}
              commute.
        \end{enumerate}
    \item $\underline\phi_{K,X,Y}$ is natural in $K$, $X$ and $Y$.
  \end{enumerate}
\end{proposition}

If a monoidal category $\V$ is closed and symmetric and a category $\S$ has a $\V$-structure with cylinder, then we know from Theorem \ref{thm:cp2.7.5} and the discussion in chapter 1.10 in \cite{kelly-05} that there exists a $\V$-functor $-\otimes -:\V\otimes\S\rightarrow \S$ whose object function is \eqref{eqn:cp2.5.3}
such that the isomorphism $\underline\phi_{K,X,Y}$ in \eqref{eqn:cp2.4.1} is $\V$-natural in $K,X,Y$ and $\alpha_{K,X}$ in \eqref{eqn:cp2.4.0} is a component of the unit of the $\V$-adjunction associated with $\underline\phi_{-,X,-}$.

Before we compare (1) an (2), we recall the tensored $\V$-category. We would like to mention that in contrast with the Definition 6.5.1 in \cite{borceux-94-2}, we do not ask $\V$ to be symmetric.

\begin{definition}[Definition 6.5.1 in \cite{borceux-94-2}]
\label{def:cp2.1.3}
  Let $\V$ be a closed monoidal category.
  We say that a $\V$-category $\S$ is tensored if
  \begin{enumerate}
    \item for every $K\in\ob\V$ and $X\in\ob\S$, there exists an object
        \begin{equation}
        \label{eqn:cp2.7.1}
          K\otimes X
        \end{equation}
        of $\S$, and
    \item for every for every $K\in\ob\V$ and $X,Y\in \ob\S$, there exists an isomorphism
        \begin{equation}
        \label{eqn:cp2.5.1}
          \underline\phi_{K,X,Y}:\uS(K\otimes X,Y)
            \rightarrow \uV(K,\uS(X,Y))
        \end{equation}
        of $\V$ such that $\underline\phi_{K,X,Y}$ is $\V$-natural in $Y$.
  \end{enumerate}
  We call $K\otimes X$ the tensor of $K$ and $X$.
\end{definition}

In Remark 2 in chapter 2.1 of \cite{quillen-67}, Quillen mentioned that if $\S$ is a category with $\sSet$-structure and cylinder, then $K\otimes X$ represents the $\sSet$-functor $\usSet(K,\uS(X,-))$. The following result shows it holds for every closed monoidal category $\V$.

\begin{theorem}
\label{thm:cp2.7.5}
  Let $\V$ be a closed monoidal category.
  Let $\S$ be a category with $\V$-structure.
  \begin{enumerate}
    \item If $\S$ has cylinder, then the $\V$-category associated with $\S$ by Proposition \ref{pro:cp2.3.4} is tensored: \eqref{eqn:cp2.7.1} and \eqref{eqn:cp2.5.1} are provided by \eqref{eqn:cp2.5.3} and \eqref{eqn:cp2.4.1} respectively and
        \eqref{eqn:cp2.4.1} is $\V$-natural in $Y$.
    \item   If the $\V$-category associated with $\S$ by Proposition \ref{pro:cp2.3.4} is tensored,
        then $\S$ has cylinder: for every $K\in\ob\V$ and $X\in\ob\S$,
        \begin{equation}
        \label{eqn:cp2.1.16}
          (K\otimes X,\varpi^{-1}(\underline\phi_{K,X,K\otimes X}\cdot j_{K\otimes X}),\{\underline\phi_{K,X,Y}\}_{Y\in\ob\S})
        \end{equation}
        is a $K$-cylinder of $X$ where $\varpi$ is \eqref{eqn:cp2.2.17}.
  \end{enumerate}
\end{theorem}

The $\V$-module was introduced by Hovey in \cite{hovey-99}
in the course of defining the closed $\V$-model category when $\V$ is a monoidal model category.
$\sSet$ is a monoidal model category.
The $\V$-modules are the generalization of the monoidal category $\V$ as the $R$-modules generalize a ring $R$ not necessarily commutative.
We would like to point out that in the following two definitions we do not assume that $\V$ is closed.

\begin{definition}[Definition 4.1.6 in \cite{hovey-99}]
\label{def:cp2.6.1}
  Let $\V$ be a monoidal category.
  We say that a category $\S$ is a left $\V$-module if
  \begin{enumerate}
    \item there exists a functor
        \begin{equation}
        \label{eqn:cp2.6.0}
          -\otimes -:\V\times\S\rightarrow \S,
        \end{equation}
    \item for every $K,L\in\ob\V$ and $X\in\ob\S$, there exists an isomorphism
        \begin{equation}
        \label{eqn:cp2.6.1}
          a_{K,L,X}:(K\otimes L)\otimes X\xrightarrow{\cong} K\otimes(L\otimes X)
        \end{equation}
        of $\S$ natural in $K,L$ and $X$, and
    \item for every $X\in\ob\S$, there exists an isomorphism
        \begin{equation}
        \label{eqn:cp2.6.2}
          l_X:I \otimes X\xrightarrow{\cong} X
        \end{equation}
        of $\S$ natural in $X$,
  \end{enumerate}
  such that
  \begin{enumerate}
  \setcounter{enumi}{3}
    \item for every $K,L,M\in\ob\V$ and $X\in\ob\S$, the following diagram commutes
        \begin{equation}
        \label{eqn:cp2.6.4}
          \xymatrix@C=1.5em{
           ((K\otimes L)\otimes M)\otimes X\ar[r]^a  \ar[d]_{a \otimes 1}& (K\otimes L)\otimes(M\otimes X)\ar[r]^a  & K\otimes(L\otimes(M\otimes X))\\
           (K\otimes (L\otimes M))\otimes X\ar[rr]_a  && K\otimes((L\otimes M)\otimes X)\ar[u]_{1\otimes a }}
        \end{equation}
        and
    \item for every $K\in\ob\V$ and $X\in\ob\S$, the following diagram commutes.
        \begin{equation}
        \label{eqn:cp2.6.5}
          \xymatrix{
            (K\otimes I )\otimes X \ar[rr]^a \ar[dr]_{r\otimes 1} && K\otimes (I \otimes X)\ar[dl]^{1\otimes l}\\
            & K\otimes X&}
        \end{equation}
  \end{enumerate}
\end{definition}

\begin{remark}
  There is a corresponding notion of a right $\V$-module.
  But we will only deal with left $\V$-modules in this note. So we call them simply $\V$-modules.
\end{remark}

\begin{remark}
  In \eqref{eqn:cp2.6.1} and \eqref{eqn:cp2.6.2}, we used the same notations as the corresponding natural transformations of $\V$.
  We hope that no confusion arises.
  By an abuse of notion, we will write $a$ and $l$ instead of $a_{K,L,X}$ and $l_{X}$ respectively.
\end{remark}

If a monoidal category $\V$ is symmetric, then every left $\V$-module has an associated right $\V$-module structure. $X\otimes K$ is defined by $K\otimes X$ and the associativity morphism of the right $\V$-module is $a_{L,K,X}\cdot (c_{K,L}\otimes 1)$ where $c_{K,L}$ is the morphism \eqref{eqn:cp2.2.19} of $\V$.

Besides the model structure, a closed $\V$-model category has one more structure, the adjunction of two variables.
Hovey call a $\V$-module with an adjunction of two variables a closed $\V$-module.
An adjunction of two variables consists of three functors and two natural isomorphisms.
Of the three functors, the two are \eqref{eqn:cp2.6.0} and \eqref{eqn:cp2.6.6}.
\eqref{eqn:cp2.6.7} is one of the two natural isomorphisms.

\begin{definition}[cf.~Definition 4.1.12 in \cite{hovey-99}]
  Let $\V$ be a monoidal category.
  We say that a $\V$-module $\S$ is tensor-closed  if
  \begin{enumerate}
    \item there exists a functor
        \begin{equation}
        \label{eqn:cp2.6.6}
          \uS(-,-):\S^{\op}\times\S\rightarrow \V
        \end{equation}
        and
    \item for every $K\in\ob\V$ and $X,Y\in\ob\S$, there exists an isomorphism
        \begin{equation}
        \label{eqn:cp2.6.7}
          \phi_{K,X,Y}:\S(K\otimes X,Y)
            \xrightarrow{\cong} \V(K,\uS(X,Y))
        \end{equation}
        natural in $K,X,Y$.
  \end{enumerate}
\end{definition}

\begin{remark}
 A monoidal category $\V$ is closed if and only if $\V$ is a tensor-closed as a $\V$-module.
\end{remark}

\begin{remark}
  Again, by an abuse of notion, we will write $\phi$ instead of $\phi_{K,X,Y}$ when no confusion is likely to occur.
\end{remark}

\begin{definition}
  Let $\V$ be a monoidal category.
  Let $\S$ be a tensor-closed $\V$-module.
  We denote the unit and the counit of the adjunction $\phi_{-,X,-}$ in \eqref{eqn:cp2.6.7} with
  $\eta$ and $\epsilon$:
  \begin{align}
    \eta_K &: K\rightarrow \uS(X,K\otimes X) \\
    \label{eqn:cp2.7.13}
    \epsilon_Y &: \uS(X,Y)\otimes X\rightarrow Y
  \end{align}
\end{definition}

Just like a closed monoidal category, a tensor-closed $\V$-module has a canonical $\V$-structure, which we explain below.

\begin{definition}
  Let $\V$ be a monoidal category. Let $\S$ be a tensor-closed $\V$-module.
  For every $X,Y,Z\in\ob\S$, we define a morphism
  \begin{equation}
  \label{eqn:cp2.7.10}
    \ub_{X,Y,Z}:\uS(Y,Z)\otimes\uS(X,Y)\rightarrow\uS(X,Z)
  \end{equation}
  of $\V$ by
  \begin{equation}
  \label{eqn:cp2.7.11}
    \ub_{X,Y,Z}=\phi_{\uS(Y,Z)\otimes\uS(X,Y),X,Z}(\epsilon_Z\cdot (1\otimes\epsilon_Y)\cdot a).
  \end{equation}
\end{definition}

\begin{definition}
\label{eqn:def:1.17}
  Let $\V$ be a monoidal category. Let $\S$ be a tensor-closed  $\V$-module.
  We denote by $\varphi$ the natural isomorphism whose component $\varphi_{X,Y}$ at $(X,Y)\in\ob(\S\times\S)$ is the composition
  \begin{equation}
  \label{eqn:cp2.7.12}
    \S(X,Y)\xrightarrow{l_X^*}\S(I\otimes X,Y)\xrightarrow{\phi_{I,X,Y}}\V(I,\uS(X,Y)).
  \end{equation}
\end{definition}

We are using $\varphi$ for \eqref{eqn:cp2.4.15} and \eqref{eqn:cp2.7.12}.
The next proposition will explain why.
It is a generalization of a property, Proposition \ref{pro:cp2.3.5}, of closed monoidal categories.

\begin{proposition}
\label{pro:cp2.7.12}
  Let $\V$ be a monoidal category.
  Let $\S$ be a tensor-closed $\V$-module.
  Then $\uS(-,-)$ in  \eqref{eqn:cp2.6.6}, $\ub_{X,Y,Z}$ in  \eqref{eqn:cp2.7.10} and $\varphi$ in Definition \ref{eqn:def:1.17} form a $\V$-structure of $\S$.
\end{proposition}

The following result shows that (1) and (3) are equivalent.
It is implicit in \cite{quillen-67}.
See Remark 1 in chapter 2.1 in \cite{quillen-67}.

\begin{theorem}
\label{thm:cp2.6.9}
  Let $\V$ be a closed monoidal category.
  Let $\S$ be a category.
  \begin{enumerate}
    \item If $\S$ is a tensor-closed $\V$-module equipped with the $\V$-structure in Proposition \ref{pro:cp2.7.12}, then the $\V$-category associated with $\S$ by Proposition \ref{pro:cp2.3.4} is tensored. Therefore by Theorem \ref{thm:cp2.7.5}, $\S$ has cylinder.
    \item The correspondence in (1) gives a bijection from the set of tensor-closed $\V$-module structures on $\S$ to  the set of $\V$-structures with cylinder on $\S$.
  \end{enumerate}
\end{theorem}

From now on, we will consider categories with $\V$-structure, cylinder and path. So, as we mentioned earlier,  we will need the symmetry of $\V$.

\begin{definition}
  If $\V$ is a symmetric monoidal category and $\S$ is a category with $\V$-structure, there is a $\V$-structure on $\S^{\op}$ such that the corresponding $\V$-category is the opposite of the $\V$-category associated to $\S$.
  We will call it the opposite $\V$-structure of $\S$.
  So, the opposite $\V$-structure of $\S$ consists of
  \begin{enumerate}
    \item a functor
        \begin{equation}
          \uS^{\op}(-,-):\S\times\S^{\op}\rightarrow \V
        \end{equation}
        defined by a composition
        \begin{equation}
          \S\times\S^{\op} \rightarrow\S^{\op}\times\S \xrightarrow{\uS(-,-)}\V,
        \end{equation}
    \item for every $X,Y,Z\in\ob\S$, a morphism
        \begin{equation}
          \ub^{\op}_{X,Y,Z}:\uS^{\op}(Y,Z)\otimes\uS^{\op}(X,Y) \rightarrow\uS^{\op}(X,Z)
        \end{equation}
        of $\V$ defined by a composition
        \begin{equation}
          \uS^{\op}(Y,Z)\otimes\uS^{\op}(X,Y) \xrightarrow{c}
          \uS^{\op}(X,Y)\otimes\uS^{\op}(Y,Z) \xrightarrow{\ub_{Z,Y,X}}
          \uS^{\op}(X,Z)
        \end{equation}
        where $c$ is \eqref{eqn:cp2.2.19} and
    \item a natural isomorphism
        \begin{equation}
          \varphi^{\op}:\S^{\op}(-,-)\xrightarrow{\cong}\V( I ,\uS^{\op}(-,-))
        \end{equation}
        defined by
        \begin{equation}
        \label{eqn:cp2.1.34}
          \varphi^{\op}_{X,Y}=\varphi_{Y,X}
        \end{equation}
        where $X,Y\in\ob\S$.
  \end{enumerate}
\end{definition}

\begin{definition}[cf.~Definition  3 in section 2.1 of \cite{quillen-67}]
\label{def:cp2.1.28}
  Let $\V$ be a closed symmetric monoidal category.
  Let $\S$ be a category with $\V$-structure. We say that $\S$ has path if $\S^{\op}$ with the opposite $\V$-structure of $\S$ has cylinder.
  So, a $K$-path of $X$ consists of
  \begin{enumerate}
    \item an object
        \begin{equation}
          K\pitchfork X
        \end{equation}
        of $\S$,
    \item a morphism
        \begin{equation}
          \beta_{K,X}:K\rightarrow \uS(K\pitchfork X,X)
        \end{equation}
        of $\V$, and
    \item for every $Y\in \ob\S$, an isomorphism
        \begin{equation}
          \underline\psi_{K,X,Y}:\uS(Y,K\pitchfork X)
            \rightarrow \uV(K,\uS(Y,X))
        \end{equation}
        of $\V$ such that the following diagram commutes.
        \begin{equation}
        \label{eqn:cp2.1.25}
          \xymatrix{
            \uS(Y,K\pitchfork X)\otimes K
              \ar[dd]_{\underline\psi_{K,X,Y}\otimes 1}
              \ar[rr]^(0.44)
                {1\otimes\beta_{K,X}}
              && \uS(Y,K\pitchfork X)\otimes \uS(K\pitchfork X,X)
              \ar[d]^{c}
            \\
              && \uS(K\pitchfork X,X)\otimes \uS(Y,K\pitchfork X)
              \ar[d]^{\ub}
            \\
            \uV(K, \uS(Y,X))\otimes K
              \ar[rr]_(0.58){\epsilon_{\uS(Y,X)}}
                && \uS(Y,X)}
        \end{equation}
  \end{enumerate}
\end{definition}

In Remark \ref{rem:cp2.1.4} we mentioned that, if $\V$ is closed, $\V$ has a $\V$-structure with cylinder. If $\V$ is also symmetric, then $\V$ has path too: let $K,L\in\ob\V$. Then $K\pitchfork L=\uV(K,L)$, $\beta_{K,L}=\pi(\epsilon_L\cdot c_{K,K\pitchfork L})$ where $\epsilon_L$ is \eqref{eqn:elno.5.3} and $c_{K,K\pitchfork L}$ is \eqref{eqn:cp2.2.19},
and $\underline\psi_{K,L,M}=\underline\phi_{K,M,L}\cdot c_{M,K}^*\cdot\underline\phi_{M,K,L}^{-1}$.
See the example following Definition 3 in section 2.1 of \cite{quillen-67}.

\begin{definition}
  Let $\V$ be a monoidal category.
  We say that a category $\S$ is a $\V$-comodule if $\S^{\op}$ is a $\V$-module. So, a $\V$-comodule structure on $\S$ consists of
  \begin{enumerate}
    \item a functor
        \begin{equation}
          -\pitchfork -:\V\times\S^{\op}\rightarrow \S^{\op},
        \end{equation}
    \item for every $K,L\in\ob\V$ and $X\in\ob\S$, an isomorphism
        \begin{equation}
          a^{\op}_{K,L,X}:(K\otimes L)\pitchfork X\xleftarrow{\cong} K\otimes(L\pitchfork X)
        \end{equation}
        of $\S$ natural in $K,L$ and $X$, and
    \item for every $X\in\ob\S$, an isomorphism
        \begin{equation}
          l^{\op}_X:I \pitchfork X\xleftarrow{\cong} X
        \end{equation}
        of $\S$ natural in $X$
  \end{enumerate}
  such that
  \begin{enumerate}
  \setcounter{enumi}{3}
    \item for every $K,L,M\in\ob\V$ and $X\in\ob\S$, the following diagram in $\S$ commutes
        \begin{equation}
        \label{eqn:cp2.1.42}
          \xymatrix@C=1.5em{
           ((K\otimes L)\otimes M)\pitchfork X
           & (K\otimes L)\pitchfork(M\pitchfork X)\ar[l]_{a^{\op}} & K\pitchfork(L\pitchfork(M\pitchfork X))\ar[l]_{a^{\op}}\ar[d]^{1\pitchfork a^{\op}} \\
           (K\otimes (L\otimes M))\pitchfork X
           \ar[u]^{a\pitchfork 1}
           && K\pitchfork((L\otimes M)\pitchfork X)\ar[ll]^{a^{\op}}}
        \end{equation}
        and
    \item for every $K\in\ob\V$ and $X\in\ob\S$, the following diagram in $\S$ commutes.
        \begin{equation}
        \label{eqn:cp2.1.43}
          \xymatrix{
            (K\otimes I )\pitchfork X  && K\pitchfork (I \pitchfork X)\ar[ll]_{a^{\op}} \\
            & K\pitchfork X \ar[ur]_{1\pitchfork l^{\op}} \ar[ul]^{r\pitchfork 1}&}
        \end{equation}
  \end{enumerate}
\end{definition}

\begin{definition}[cf.~Definition 4.1.12 in \cite{hovey-99}]
  Let $\V$ be a monoidal category.
  We say that a $\V$-comodule $\S$ is cotensor-closed  if
  $\S^{\op}$ is tensor-closed.
  So a $\V$-comodule $\S$ is cotensor-closed if
  \begin{enumerate}
    \item there exists a functor
        \begin{equation}
          \uS^{\op}(-,-):\S\times\S^{\op}\rightarrow \V
        \end{equation}
        and
    \item for every $K\in\ob\V$ and $X,Y\in\ob\S$, there exists an isomorphism
        \begin{equation}
          \psi_{K,X,Y}:\S(Y, K\pitchfork X)
            \xrightarrow{\cong} \V(K,\uS(Y,X))
        \end{equation}
        of $\V$ natural in $K,X,Y$.
  \end{enumerate}
\end{definition}

Suppose that $\S$ and $\S^{\op}$ are tensor-closed $\V$-modules. Considering the definition of the simplicial model category of Quillen, it seems reasonable to ask the $\V$-structure of $\S^{\op}$ to be the opposite $\V$-structure of $\S$.

\begin{definition}
  Let $\V$ be a monoidal category.
  We say that a category $\S$ is closed $\V$-bimodule if $\S$ and $\S^{\op}$ are tensor-closed $\V$-module and the $\V$-structure of $\S^{\op}$ is the opposite $\V$-structure of $\S$.
\end{definition}

Our last result shows that if $\V$ is closed and symmetric, given a closed $\V$-bimodule $\S$, the associativity natural isomorphism $a^{\op}$ and the unit natural isomorphism $l^{\op}$ of $\S^{\op}$ become redundant and can be recovered from the rest of the structure, the closed $\V$-module structure of Hovey.

In the following definition, we do not assume that $\V$ is closed nor symmetric.

\begin{definition}[cf.~Definition 4.1.13 in \cite{hovey-99}]
  Let $\V$ be a monoidal category.
  Suppose that $\S$ is a tensor-closed left $\V$-module.
  We say that $\S$ is a closed $\V$-module if
  \begin{enumerate}
    \item there exists a functor
        \begin{equation}
        \label{eqn:cp2.1.14}
          -\pitchfork - :\V\times\S^{\op}\rightarrow \S^{\op}
        \end{equation}
        and
    \item for every $K\in\ob\V$ and $X,Y\in\ob\S$, there exists an isomorphism
        \begin{equation}
        \label{eqn:cp2.1.15}
          \psi_{K,X,Y}:\S(Y,K\pitchfork X)
            \xrightarrow{\cong} \V(K,\uS(Y,X))
        \end{equation}
        of $\V$ natural in $K,X,Y$.
  \end{enumerate}
\end{definition}

The previous definition is different from  but equivalent to that of Hovey. It is more convenient for us to prove the following theorem.

\begin{theorem}
\label{thm:cp2.1.11}
  Let $\V$ be a closed symmetric monoidal category.
  Let $\S$ be a category.
  If $\S$ is a closed $\V$-module, then there exists an unique closed $\V$-bimodule structure on $\S$ extending the closed $\V$-module structure of $\S$.
  In particular, there is a bijection from the set of $\V$-structures with cylinder and path on $\S$ to the set of closed $\V$-module structures on $\S$.
\end{theorem}

We close this section with one remark on commutative diagrams in this note.
For simplicity, when we draw commutative diagrams, we generally omit subdiagrams representing coherence as well as corresponding parenthesis.

\section{Review on Monoidal Category}

In this section, we recall some properties of monoidal categories and set some notations. Our main reference is the chapter 1 of \cite{kelly-05}.

We will denote a monoidal category and its underlying category with the same symbol. So a monoidal category $\V$ consists of a category $\V$, a functor $\otimes:\V\times\V\rightarrow\V$, an object $I$ of $\V$ and natural transformations $a_{X,Y,Z}:(X\otimes Y)\otimes Z\rightarrow X\otimes(Y\otimes Z)$, $l_X:I\otimes X\rightarrow X$, $r_X:X\otimes I\rightarrow X$ subject to two coherence axioms: the commutativity of
\begin{equation}
\label{eqn:elno.1.1}
  \xymatrix{
   ((W\otimes X)\otimes Y)\otimes Z\ar[r]^{a} \ar[d]_{{a}\otimes 1}& (W\otimes X)\otimes(Y\otimes Z)\ar[r]^{a} & W\otimes(X\otimes(Y\otimes Z))\\
   (W\otimes (X\otimes Y))\otimes Z\ar[rr]_{a} && W\otimes((X\otimes Y)\otimes Z)\ar[u]_{1\otimes{a}}}
\end{equation}
and
\begin{equation}
\label{eqn:elno.1.2}
  \xymatrix{
    (X\otimes I)\otimes Y \ar[rr]^{a} \ar[dr]_{r\otimes 1} && X\otimes (I\otimes Y)\ar[dl]^{1\otimes l}\\
    & X\otimes Y&}
\end{equation}

From now on, we fix a monoidal category $\V=(\V,\otimes,I,a,l,r)$.

\begin{lemma}
  Let $\V$ be a monoidal category. Then
  \begin{equation}
    r_I=l_I:I\otimes I\rightarrow I
  \end{equation}
  holds and the diagram
  \begin{equation}
  \label{eqn:cp2.2.6}
    \xymatrix{
      (I\otimes X)\otimes Y \ar[rr]^a \ar[dr]_{l\otimes 1} && I\otimes (X\otimes Y)\ar[dl]^{l}\\
      & X\otimes Y&}
  \end{equation}
  commutes for every $X,Y\in\ob\V$.
\end{lemma}
\begin{proof}
 It follows from \eqref{eqn:elno.1.1} and \eqref{eqn:elno.1.2}.
 See Proposition 1.1 in \cite{joyal-street-93} for the detail.
\end{proof}

\begin{definition}
  We say that a monoidal category $\V$ is closed if,
  for every $Y\in\ob\V$, the functor $-\otimes Y:\V\rightarrow\V$ has a right adjoint
  \begin{equation}
  \label{eqn:cp2.2.25}
    \uV(Y,-):\V\rightarrow\V.
  \end{equation}
  We denote by $\pi$ the natural isomorphism associated with this adjunction:
  \begin{equation}
  \label{eqn:elno.5.1}
    \pi_{X,Y,Z}:\V(X\otimes Y,Z)\xrightarrow{\cong} \V(X,\uV(Y,Z))
  \end{equation}
  We denote the unit and the counit of $\pi_{-,Y,-}$ with
  $\eta$ and $\epsilon$:
  \begin{align}
    \label{eqn:cp2.2.2}
    \eta_X &: X\rightarrow \uV(Y,X\otimes Y) \\
    \label{eqn:elno.5.3}
    \epsilon_Z &: \uV(Y,Z)\otimes Y\rightarrow Z
  \end{align}
\end{definition}

\begin{lemma}
  Let $\V$ be a closed monoidal category. Then for every $f\in\V(X,Y)$ and $Z\in\ob\V$,
  the following diagram commutes.
  \begin{equation}
  \label{eqn:cp2.2.14}
    \xymatrix{
      \uV(Y,Z)\otimes X \ar[r]^{1\otimes f} \ar[d]_{\uV(f,Z)\otimes 1} & \uV(Y,Z)\otimes Y \ar[d]^{\epsilon^Y_Z}\\
      \uV(X,Z)\otimes X \ar[r]_{\epsilon^X_Z} & Z}
  \end{equation}
\end{lemma}
\begin{proof}
  It follows from the naturality of \eqref{eqn:elno.5.1} in $X$ and $Y$.
\end{proof}

\begin{lemma}
  Let $\V$ be a closed monoidal category.
  Then the parametrized family $\{\uV(Y,-)\}_{Y\in\ob\V}$ of functors in \eqref{eqn:cp2.2.25} extends to a bifunctor
  \begin{equation}
  \label{eqn:cp2.2.15}
    \uV(-,-):\V^{\op}\times\V\rightarrow\V
  \end{equation}
  such that $\pi_{X,Y,Z}$ in \eqref{eqn:elno.5.1} is natural in $Y$.
\end{lemma}
\begin{proof}
  Theorem 4.7.3 in \cite{maclane-98}.
\end{proof}

\begin{definition}
  Let $\V$ be a closed monoidal category.
  Then for every $X,Y,Z\in\ob\V$, we define a morphism
  \begin{equation}
  \label{eqn:cp2.2.16}
    \ub_{X,Y,Z}:\uV(Y,Z)\otimes\uV(X,Y)\rightarrow\uV(X,Z)
  \end{equation}
  of $\V$ by
  \begin{equation}
  \label{eqn:cp2.2.4}
    \ub_{X,Y,Z}=\pi(\epsilon_Z\cdot (1\otimes\epsilon_Y)\cdot a).
  \end{equation}

\end{definition}

\begin{definition}
\label{def:cp2.2.6}
  Let $\V$ be a closed monoidal category.
  Then we denote by $\varpi$ the natural isomorphism whose component $\varpi_{X,Y}$ at $(X,Y)$ is a composition
  \begin{equation}
  \label{eqn:cp2.2.17}
    \V(X,Y)\xrightarrow{l^*_X} \V(I\otimes X, Y)\xrightarrow{\pi_{I,X,Y}} \V(I,\uV(X,Y)).
  \end{equation}
\end{definition}

The next proposition says that if a monoidal category $\V$ is closed,
$\V$ has a $\V$-structure.
Later in section \ref{sec:cp2.tcvm}, we will prove a more general fact, Proposition \ref{pro:cp2.7.12}.
We note that the role of $\varphi$ is played by $\varpi$.

\begin{proposition}
\label{pro:cp2.3.5}
  Let $\V$ be a closed monoidal category.
  Then $\uV(-,-)$ in \eqref{eqn:cp2.2.15}, $\ub_{X,Y,Z}$ in \eqref{eqn:cp2.2.16} and $\varpi$ in Definition \ref{def:cp2.2.6} form a $\V$-structure of $\V$.
\end{proposition}

\begin{remark}
  In \cite{kelly-05}, the $\V$-structure on $\V$ is discussed under an additional assumption, namely, the symmetry. If we want to have a $\V$-functor $\uV(-,-):\V^{\op}\otimes\V\rightarrow\V$ make sense, we need the symmetry. However just to define an  ordinary functor $\uV(-,-):\V^{\op}\times\V\rightarrow\V$, we do not need the symmetry of $\V$.
  The latter functor is what we need to show that a closed monoidal category $\V$ is a category with $\V$-structure.
  This will be important for us later in proofs.
\end{remark}

\begin{definition}
  Let $\V$ be a closed monoidal category.
  We denote by $\underline{\pi}$ the natural isomorphism  whose component
  \begin{equation}
  \label{eqn:cp2.2.3}
    \underline{\pi}_{X,Y,Z}:\uV(X\otimes Y,Z)\xrightarrow{\cong} \uV(X,\uV(Y,Z))
  \end{equation}
  at $(X,Y,Z)$ is characterized by the following commutative diagram: let $W\in\ob\V$.
  \begin{equation}
  \label{eqn:cp2.2.1}
    \xymatrix{
      \V((W\otimes X)\otimes Y,Z)\ar[r]^\pi & \V(W\otimes X,\uV(Y, Z))\ar[r]^\pi &
      \V(W,\uV(X,\uV(Y,Z)))\\
      \V(W\otimes (X\otimes Y),Z)\ar[u]^{a^*} \ar[rr]_\pi && \V(W,\uV(X\otimes Y,Z))\ar[u]_{\underline{\pi}_{X,Y,Z}}}
  \end{equation}
\end{definition}

\begin{lemma}
\label{lem:cp2.2.10}
  Let $\V$ be a closed monoidal category.
  Then the following diagram commutes.
  \begin{equation}
    \xymatrix{
      \V(X\otimes Y,Z) \ar[rr]^(0.47){\varpi_{X\otimes Y, Z}} \ar[d]_{\pi_{X,Y,Z}} && \V(I,\uV(X\otimes Y, Z)) \ar[d]^{\V(I,\underline{\pi}_{X,Y,Z})}\\
      \V(X,\uV(Y,Z)) \ar[rr]_(0.47){\varpi_{X,\uV(Y,Z)}} && \V(I,\uV(X,\uV(Y,Z)))}
  \end{equation}
\end{lemma}
\begin{proof}
  Replacing $W$ with $I$ in \eqref{eqn:cp2.2.1} and use \eqref{eqn:cp2.2.6}.
\end{proof}

\begin{definition}
  Let $\V$ be a closed monoidal category.
  We denote by $i$ the natural isomorphism whose component
  \begin{equation}
  \label{eqn:cp2.2.9}
    i_X:X\xrightarrow{\cong} \uV(I,X)
  \end{equation}
  at $X$ is defined by
  \begin{equation}
  \label{eqn:cp2.2.20}
    i_X=\pi(r_X).
  \end{equation}
\end{definition}

\begin{lemma}
  Let $\V$ be a closed monoidal category.
  If $X,Y\in\ob\V$ and $f\in\V(X,Y)$, then the following diagram commutes.
  \begin{equation}
  \label{eqn:cp2.2.18}
    \xymatrix{
      \V(X\otimes I,Y) \ar[r]^\pi &
      \V(X,\uV(I,Y))\\
      & \V(X,Y) \ar[u]_{i_Y} \ar[ul]^{r_X^*}}
  \end{equation}
\end{lemma}
\begin{proof}
  It follows from $i_Y\cdot f=\pi(r_Y\cdot f\otimes 1)=\pi(f\cdot r_X)$.
\end{proof}

\begin{definition}
  We say that a monoidal category $\V$ is symmetric if for every $X,Y\in\V$, there exists a natural isomorphism
  \begin{equation}
    \label{eqn:cp2.2.19}
    c_{X,Y}:X\otimes Y\rightarrow Y\otimes X
  \end{equation}
  such that the following three diagrams commute.
  \begin{equation}
  \label{eqn:elno.4.11}
    \xymatrix{
      X\otimes Y\ar[r]^{c} \ar[dr]_1 & Y\otimes X \ar[d]^{c}\\
      & X\otimes Y}
  \end{equation}
  \begin{equation}
  \label{eqn:elno.4.12}
    \xymatrix{
      (X\otimes Y)\otimes Z\ar[r]^{a} \ar[d]_{c\otimes 1} &X\otimes(Y\otimes Z) \ar[r]^c & (Y\otimes Z)\otimes X\ar[d]^{a}\\
      (Y\otimes X)\otimes Z \ar[r]_{a} & Y\otimes(X\otimes Z) \ar[r]_{1\otimes c}& Y\otimes (Z\otimes X)}
  \end{equation}
  \begin{equation}
  \label{eqn:elno.4.13}
    \xymatrix{
      I\otimes X\ar[rr]^{c} \ar[dr]_l && X\otimes I \ar[dl]^r\\
      & X}
  \end{equation}
\end{definition}

\section{Review on $\V$-Category}

In this section, we recall some properties of $\V$-categories. Our main reference is again \cite{kelly-05}. But we use ``adjective'' form for the notations of $\V$-categories. So a $\V$-category and its underlying category will be denoted by the same symbol.

First, we recall the $\V$-category, the $\V$-functor and the $\V$-natural transformation.

\begin{definition}[\cite{kelly-05}]
\label{def:cp2.4.3}
  Let $\V$ be a monoidal category.
  A $\V$-category $\A$ consists of
  \begin{enumerate}
    \item a set $\ob\A$ of objects,
    \item for every $A,B\in\ob\A$, an object
        \begin{equation}
        \label{eqn:cp2.3.0}
          \uA(A,B)
        \end{equation}
        of $\V$,
    \item for every $A,B,C\in\ob\uA$,  a morphism
        \begin{equation}
        \label{eqn:cp2.4.14}
          \ub_{A,B,C}:\uA(B,C)\otimes\uA(A,B)\rightarrow \uA(A,C)
        \end{equation}
        of $\V$ and
    \item for every $A\in\ob\A$, a morphism
        \begin{equation}
          j_A:I\rightarrow \uA(A,A)
        \end{equation}
        of $\V$
  \end{enumerate}
such that
  \begin{enumerate}
  \setcounter{enumi}{4}
    \item for every $A,B,C,D\in\ob\A$, the following diagram commutes
        \begin{equation}
        \label{eqn:elno.2.2}
          \xymatrix@C=-1em{
            (\uA(C,D)\otimes\uA(B,C))\otimes\uA(A,B) \ar[rr]^{a}
              \ar[d]_{\ub\otimes 1} &
              &\uA(C,D)\otimes(\uA(B,C)\otimes\uA(A,B))
              \ar[d]^{1 \otimes\ub}\\
            \uA(B,D)\otimes\uA(A,B) \ar[r]_(0.6){\ub} & \uA(A,D)
              & \uA(C,D)\otimes\uA(A,C) \ar[l]^(0.6){\ub}}
        \end{equation}
        and
    \item for every $A,B\in\ob\A$, the following diagram commutes.
        \begin{equation}
        \label{eqn:elno.2.3}
          \xymatrix{
            \uA(B,B)\otimes\uA(A,B) \ar[r]^(0.6){\ub}& \uA(A,B)
              & \uA(A,B)\otimes\uA(A,A) \ar[l]_(0.6){\ub}\\
            I\otimes\uA(A,B) \ar[u]^{j_B\otimes 1} \ar[ur]_{l}&
              & \uA(A,B)\otimes I\ar[ul]^{r} \ar[u]_{1 \otimes j_A}\\}
        \end{equation}
  \end{enumerate}
\end{definition}

\begin{definition}
\label{def:cp2.3.2}
  Let $\V$ be a monoidal category.
  Let $\A$ and $\B$ be  $\V$-categories.
  A $\V$-functor $T:\A\rightarrow \B$ consists of
  \begin{enumerate}
    \item a function $T:\ob \A\rightarrow\ob\B$ and
    \item for every $A,B\in\ob\A$, a morphism $T_{A,B}:\uA(A,B)\rightarrow\uB(TA,TB)$ of $\V$
  \end{enumerate}
  such that
  \begin{enumerate}
  \setcounter{enumi}{2}
    \item for every $A,B,C\in\ob\A$, the following diagram commutes
        \begin{equation}
        \label{eqn:elno.2.6}
          \xymatrix{
            \uA(B,C)\otimes\uA(A,B) \ar[rr]^\ub \ar[d]_{T_{B,C}\otimes T_{A,B}}
              && \uA(A,C) \ar[d]^{T_{A,C}}\\
            \uB(TB,TC)\otimes\uB(TA,TB) \ar[rr]_\ub
              && \uB(TA,TC)}
        \end{equation}
        and
    \item for every $A\in\ob\A$, the following diagram commutes.
        \begin{equation}
        \label{eqn:elno.2.7}
          \xymatrix{
            I \ar[r]^{j_A} \ar[rd]_{j_{TA}}& \uA(A,A) \ar[d]^{T_{A,A}}\\
            &\uB(TA,TA)}
        \end{equation}
  \end{enumerate}
\end{definition}

\begin{definition}
  Let $\V$ be a monoidal category.
  Let $\A,\B$ be $\V$-categories.
  Let $S,T:\A\rightarrow\B$ be $\V$-functors.
  A $\V$-natural transformation $\alpha:S\rightarrow T$ is a collection \begin{equation*}
    \{\alpha_{A}:I\rightarrow\uB(SA,TA)\mid A\in\ob\A\}
  \end{equation*}
  of  morphisms of $\V$ such that
  for every $A,B\in\ob\A$, the following diagram commutes.
  \begin{equation}
  \label{eqn:cp2.2.8}
    \xymatrix@R=2em{
      I\otimes\uA(A,B) \ar[rr]^(0.4){\alpha_B\otimes S_{A,B}}
      && \uB(SB,TB)\otimes\uB(SA,SB) \ar[d]^\ub\\
      \uA(A,B)\ar[u]^{l^{-1}} \ar[d]_{r^{-1}} &&\uB(SA,TB)\\
      \uA(A,B)\otimes I \ar[rr]_(0.4){T_{A,B}\otimes\alpha_A}
      && \uB(TA,TB)\otimes\uB(SA,TA)\ar[u]_\ub}
  \end{equation}
\end{definition}

\begin{lemma}
\label{lem:cp2.4.14}
  Let $\V$ be a closed monoidal category so that it is a $\V$-category by Proposition \ref{pro:cp2.3.5} and Proposition \ref{pro:cp2.3.4}.
  Let $\A$ be a $\V$-category.
  Let $S,T:\A\rightarrow \V$ be $\V$-functors.
  Then $\{\alpha_{A}:I\rightarrow\uV(SA,TA)\mid A\in\ob\A\}$ is a $\V$-natural transformation from $S$ to $T$ if and only if
  for every $A,B\in\ob\A$, the following diagram commutes.
  \begin{equation}
    \xymatrix@C=5em{
      \uA(A,B)
      \ar[r]^{S_{A,B}}
      \ar[d]_{T_{A,B}}
      & \uV(SA,SB)
      \ar[d]^{\uV(1,\varpi^{-1}(\alpha_B))}
      \\
      \uV(TA,TB)
      \ar[r]_{\uV(\varpi^{-1}(\alpha_A),1)}
      & \uV(SA,TB)}
  \end{equation}
\end{lemma}
\begin{proof}
  It follows from \eqref{eqn:cp2.4.22} and \eqref{eqn:cp2.4.24}.
\end{proof}

\begin{definition}
  Let $\V$ be a closed monoidal category so that it is a $\V$-category by Proposition \ref{pro:cp2.3.5} and Proposition \ref{pro:cp2.3.4}.
  Let $\A$ be a $\V$-category.
  Let $A,B,C\in\ob\A$.
  We define a $\V$-functor
  \begin{equation}
    \uA(A,-):\A\rightarrow\V
  \end{equation}
  by
  \begin{equation}
  \label{eqn:cp2.2.5}
    \uA(A,-)_{B,C}=\pi(\ub_{A,B,C})
  \end{equation}
  where $\pi$ is the adjunction in \eqref{eqn:elno.5.1}.
\end{definition}

\begin{definition}
  Let $\V$ be a symmetric monoidal category.
  Let $\A$ be a $\V$-category.
  Then a $\V$-category $\A^{\op}$, called the opposite of $\A$, consists of
  \begin{enumerate}
    \item a set $\ob\A^{\op}$ of objects defined by
        \begin{equation}
          \ob\A^{\op}=\ob\A,
        \end{equation}
    \item for every $A,B\in\ob\A$, an object $\uA^{\op}(A,B)$ of $\V$
        defined by
        \begin{equation}
          \uA^{\op}(A,B)=\uA(B,A),
        \end{equation}
    \item for every $A,B,C\in\ob\uA^{\op}$,  a morphism
        \begin{equation}
          \ub^{\op}_{A,B,C}:\uA^{\op}(B,C)\otimes\uA^{\op}(A,B)\rightarrow \uA^{\op}(A,C)
        \end{equation}
        of $\V$
        defined by the composition
        \begin{equation}
          \uA^{\op}(B,C)\otimes\uA^{\op}(A,B) \xrightarrow{c}
          \uA^{\op}(A,B)\otimes\uA^{\op}(B,C) \xrightarrow{\ub_{C,B,A}}
          \uA^{\op}(A,C)
        \end{equation}
        where $c$ is \eqref{eqn:cp2.2.19}, and
    \item for every $A\in\ob\A$, a morphism
        \begin{equation}
          j^{\op}_A:I\rightarrow \uA^{\op}(A,A)
        \end{equation}
        of $\V$ defined by
        \begin{equation}
          j^{\op}_A=j_A.
        \end{equation}
  \end{enumerate}
\end{definition}

\begin{remark}
  If $\V$ is closed and symmetric, then for every $A\in\ob\A$, we can define a $\V$-functor $\uA(-,A)=\uA^{\op}(A,-):\A^{\op}\rightarrow\V$.
\end{remark}

\section{Category with  $\V$-Structure}
\label{sec:cp2.v-st}

Here, we show that two definitions, Definition \ref{def:cp2.4.5} and Definition \ref{def:cp2.4.3} are equivalent.

\begin{proposition}
\label{pro:cp2.3.3}
  Let $\V$ be a monoidal category.
  Let $\A$ is $\V$-category. Then there is a category  $\S$  with  $\V$-structure
  such that
  \begin{enumerate}
    \item $\ob\S=\ob\A$,
    \item for every $A,B\in\ob\S$
        \begin{equation}
        \label{eqn:cp2.4.6}
          \S(A,B)=\V(I,\uA(A,B)),
        \end{equation}
    \item for every $f\in\V(I,\uA(A,B))$ and $g\in\V(I,\uA(B,C))$, the composition $g\cdot f$ is
        \begin{equation}
        \label{eqn:cp2.4.13}
          I \xrightarrow{l^{-1}_I} I \otimes I \xrightarrow{g\otimes f} \uA(B,C)\otimes\uA(A,B) \xrightarrow{\ub}\uA(A,C),
        \end{equation}
    \item for every $A\in\ob\S$
        \begin{equation}
        \label{eqn:cp2.4.10}
          1_A=j_A,
        \end{equation}
    \item for every $A,B\in\ob\S$
        \begin{equation}
          \uS(A,B)=\uA(A,B),
        \end{equation}
    \item the morphism $\ub_{X,Y,Z}$ in \eqref{eqn:cp2.4.18} is \eqref{eqn:cp2.4.14}, and
    \item the natural isomorphism $\varphi$ in \eqref{eqn:cp2.4.15} is given by the identity \eqref{eqn:cp2.4.6}.
  \end{enumerate}
\end{proposition}
\begin{proof}
  By \eqref{eqn:elno.2.2}, the composition \eqref{eqn:cp2.4.13} is associative.
  By \eqref{eqn:elno.2.3}, the bottom of the following diagrams are identities.
  \begin{equation*}
    \xymatrix{
       I \ar[r]^{l^{-1}}\ar[d]_f& I \otimes I \ar[d]^f&&\\
      \uA(A,B)\ar[r]^{l^{-1}}& I \otimes\uA(A,B)\ar[r]^(0.4){j_A}& \uA(B,B)\otimes\uA(A,B)\ar[r]^(0.6)\ub & \uA(A,B)\\}
  \end{equation*}
  \begin{equation*}
    \xymatrix{
       I \ar[r]^{r^{-1}}\ar[d]_f& I \otimes I \ar[d]^f&&\\
      \uA(A,B)\ar[r]^{r^{-1}}&\uA(A,B)\otimes I \ar[r]^(0.4){j_A} &\uA(A,B)\otimes\uA(A,A)\ar[r]^(0.6)\ub & \uA(A,B)\\}
  \end{equation*}
  So the unit law also holds. Therefore $\S$ is a category.

  For every $f\in\S(A,B)$ and $g\in\S(B,C)$,
  we define $\uS(f,C)$ by the composition
  \begin{equation}
  \label{eqn:cp2.4.11}
    \uS(B,C)\xrightarrow{r^{-1}} \uS(B,C)\otimes I \xrightarrow{1\otimes f} \uS(B,C)\otimes\uS(A,B) \xrightarrow{\ub} \uS(A,C)
  \end{equation}
  and $\uS(A,g)$ by the composition
  \begin{equation}
  \label{eqn:cp2.4.12}
    \uS(A,B)\xrightarrow{l^{-1}} I\otimes\uS(A,B) \xrightarrow{g \otimes 1} \uS(B,C)\otimes\uS(A,B) \xrightarrow{\ub} \uS(A,C).
  \end{equation}
  By \eqref{eqn:elno.2.3}, \eqref{eqn:cp2.4.10}, \eqref{eqn:cp2.4.11} and \eqref{eqn:cp2.4.12}, $\uS(-,-)$ preserves the identities.
  Let $A\in\ob\S$, $f\in\S(B,C)$ and $g\in\S(C,D)$. Then by \eqref{eqn:cp2.4.12} and \eqref{eqn:cp2.4.13}, $\uS(A,g\cdot f)$ is the diagonal of the following commutative diagram.
  \begin{equation*}
  \resizebox{0.95\textwidth}{!}{
    \xymatrix@C=1.5em{
      \uA(A,B) \ar[r]^{l^{-1}} &  I \otimes\uA(A,B) \ar[r]^(0.45)f \ar[d]^{l^{-1}} & \uA(B,C)\otimes\uA(A,B) \ar[r]^(0.57)\ub \ar[d]^{l^{-1}} & \uA(A,C) \ar[d]^{l^{-1}}\\
      &  I \otimes I \otimes\uA(A,B) \ar[r]^(0.45)f & I \otimes\uA(B,C)\otimes\uA(A,B) \ar[r]^(0.57)\ub \ar[d]^g & I \otimes\uA(A,C) \ar[d]^g\\
      && \uA(C,D)\otimes\uA(B,C)\otimes\uA(A,B) \ar[r]^(0.57)\ub \ar[d]^\ub & \uA(C,D)\otimes\uA(A,C) \ar[d]^\ub\\
      && \uA(B,D)\otimes\uA(A,B) \ar[r]^(0.57)\ub & \uA(A,D)}}
  \end{equation*}
  Therefore $\uS(A,g\cdot f)=\uS(A,g)\cdot\uS(A,f)$.
  Similarly, if $f\in\S(A,B)$, $g\in\S(B,C)$ and $D\in\ob\S$,  then $\uS(g\cdot f,D)=\uS(f,D)\cdot\uS(g,D)$ holds.
  Let $f\in\S(A,B)$ and $g\in\S(C,D)$. Then the following diagram commutes.
  \begin{equation*}
  \resizebox{0.95\textwidth}{!}{
    \xymatrix@C=1.5em{
      \uA(B,C) \ar[r]^{l^{-1}} \ar[d]^{r^{-1}} &  I \otimes\uA(B,C) \ar[r]^(0.45)g \ar[d]^{r^{-1}} & \uA(C,D)\otimes\uA(B,C) \ar[r]^(0.57)\ub \ar[d]^{r^{-1}} & \uA(B,D) \ar[d]^{r^{-1}}\\
      \uA(B,C)\otimes I  \ar[r]^{l^{-1}} \ar[d]^f & I \otimes\uA(B,C)\otimes I  \ar[r]^(0.45)g \ar[d]^f & \uA(C,D)\otimes\uA(B,C)\otimes I  \ar[r]^(0.57)\ub \ar[d]^f & \uA(B,D)\otimes I  \ar[d]^f\\
      \uA(B,C)\otimes\uA(A,B) \ar[r]^{l^{-1}} \ar[d]^\ub & I \otimes\uA(B,C)\otimes\A(A,B) \ar[r]^(0.45)g \ar[d]^\ub & \uA(C,D)\otimes\uA(B,C)\otimes\uA(A,B) \ar[r]^(0.57)\ub \ar[d]^\ub & \uA(B,D)\otimes\uA(A,B) \ar[d]^\ub\\
      \uA(A,C) \ar[r]^{l^{-1}} &  I \otimes\uA(A,C) \ar[r]^(0.45)g & \uA(C,D)\otimes\uA(A,C) \ar[r]^(0.57)\ub & \uA(A,D)}}
  \end{equation*}
  Therefore, $\uS(-,-)$ is a bifunctor.

  Finally, since $\varphi$ is the identity, \eqref{eqn:cp2.4.22} and \eqref{eqn:cp2.4.24} hold by  \eqref{eqn:cp2.4.11} and \eqref{eqn:cp2.4.12}.
\end{proof}

\begin{definition}
  The ordinary category $\S$ in Proposition \ref{pro:cp2.3.3} (without $\V$-structure) is called the underlying category of $\A$.
\end{definition}

\begin{proposition}
\label{pro:cp2.3.4}
  Let $\V$ be a monoidal category.
  If $\S$ is a category with $\V$-structure, then
  there is a $\V$-category $\A$  such that
  \begin{enumerate}
    \item $\ob\A=\ob\S$,
    \item for every $A,B\in\ob\A$
        \begin{equation}
          \uA(A,B)=\uS(A,B),
        \end{equation}
    \item the morphism $\ub_{A,B,C}$ in \eqref{eqn:cp2.4.14} is \eqref{eqn:cp2.4.18} and
    \item for every $A\in\ob\A$
        \begin{equation}
        \label{eqn:cp2.4.8}
          j_A=\varphi(1_A).
        \end{equation}
  \end{enumerate}
\end{proposition}
\begin{proof}
  We only need to show the unit law \eqref{eqn:elno.2.3} holds.
  Since $\uS(-,-)$ is a functor $\uS(A,1_B)=1_{\uS(A,B)}=\uS(1_A,B)$. So, from \eqref{eqn:cp2.4.22} and  \eqref{eqn:cp2.4.24}, \eqref{eqn:elno.2.3} holds.
\end{proof}

\begin{lemma}
\label{lem:cp2.4.4}
  Let $\V$ be a monoidal category.
  Let $\S$ be a category with $\V$-structure.
  Let $\A$ be the $\V$-category  associated to $\S$ by Proposition \ref{pro:cp2.3.4}.
  Then, for every $f\in\S(A,B)$ and $g\in\S(B,C)$,
  \begin{equation}
  \label{eqn:cp2.4.9}
    \varphi(g)\cdot\varphi(f)=\varphi(g\cdot f)
  \end{equation}
  holds where the left composition is the composition \eqref{eqn:cp2.4.13} of the underlying category of $\A$.
\end{lemma}
\begin{proof}
  By \eqref{eqn:cp2.4.13}, $\varphi(g)\cdot \varphi(f)$ is the composition
  \begin{equation*}
    \i\xrightarrow{}\i\otimes\i\xrightarrow{\varphi(g)\otimes\varphi(f)} \uS(B,C)\otimes\uS(A,B) \xrightarrow{\ub}\uS(A,C).
  \end{equation*}
  Since $\varphi(g)\cdot \varphi(f)$ can be decomposed as
  \begin{equation*}
    \xymatrix{
      \i\ar[r]^{l^{-1}}\ar[d]_{\varphi(f)}&\i\otimes\i\ar[d]^{\varphi(f)}&&\\
      \uS(A,B)\ar[r]^{l^{-1}}&\i\otimes\uS(A,B)\ar[r]^(0.4){\varphi(g)}& \uS(B,C)\otimes\uS(A,B)\ar[r]^(0.6)\ub & \uS(A,C),\\}
  \end{equation*}
  the condition (6) of Definition \ref{def:cp2.4.5} and the commutativity of the diagram
  \begin{equation*}
    \xymatrix{
      \V(\i,\uS(A,B))\ar[rr]^{\V(\i,\uS(A,g))} && \V(\i,\uS(A,C))\\
      \S(A,B)\ar[rr]_{\S(A,g)}\ar[u]^\varphi && \S(A,C)\ar[u]_\varphi}
  \end{equation*}
  imply that \eqref{eqn:cp2.4.9} holds.
\end{proof}

\begin{remark}
  The correspondence in Proposition \ref{pro:cp2.3.3} is injective.
  Lemma \ref{lem:cp2.4.4} tells us that a category $\S$ with $\V$-structure and the underlying category of the $\V$-category $\A$ obtained by Proposition \ref{pro:cp2.3.4} from $\S$ are isomorphic as ordinary categories.
  It also clear that $\S$ and the category with $\V$-structure obtained by Proposition \ref{pro:cp2.3.3} from $\A$ have the same $\V$-structure.
  In short, every category with $\V$-structure is isomorphic to another category with the same $\V$-structure induced from an unique $\V$-category.
\end{remark}

In this section, we have used different notations, $\S$ and $\A$, for a category with $\V$-structure and the $\V$-category associated with it. But as we saw in the previous remark, it is not necessary to distinguish them. So from now on, we will use the same notation for them without mentioning it explicitly as we do in the following

\begin{remark}
  Let $\V$ be a monoidal category.
  Let $\S$ be a category with $\V$-structure.
  Let $g\in\S(B,C)$.
  If $\V$ is closed, then \eqref{eqn:cp2.4.24}
  is $\varpi^{-1}(\uS(A,-)_{B,C}\cdot \varphi(g))$ by \eqref{eqn:cp2.2.5}.
  Therefore,
  \begin{equation*}
    \varpi(\uS(A,g))=\uS(A,-)_{B,C}\cdot \varphi(g)
  \end{equation*}
  holds.
\end{remark}

\section{$\V$-Category with Tensor}

Here we will prove Theorem \ref{thm:cp2.7.5}.
First, we need to reformulate the $\V$-naturality of ${\underline\phi}_{K,X,Y}$ in $Y$.

\begin{lemma}
\label{lem:cp2.5.4}
  Let $\V$ be a closed monoidal category so that it is a $\V$-category by Proposition \ref{pro:cp2.3.5} and Proposition \ref{pro:cp2.3.4}.
  Let $\S$ be a $\V$-category.
  Suppose that for every $K\in\ob\V$ and $X,Y\in \ob\S$, there exists an isomorphism
  \begin{equation*}
    {\underline\phi}_{K,X,Y}:\uS(K\otimes X,Y)
      \rightarrow \uV(K,\uS(X,Y)).
  \end{equation*}
  Then ${\underline\phi}_{K,X,Y}$ is $\V$-natural in $Y$ if and only if for every $Z\in\ob\S$, the diagram
  \begin{equation}
  \label{eqn:cp2.4.3}
    \xymatrix{
      \uS(Y,Z)\otimes\uS(K\otimes X,Y) \ar[rr]^(0.55)\ub \ar[d]_{1\otimes{\underline\phi}_{K,X,Y}} && \uS(K\otimes X,Z)\ar[d]^{{\underline\phi}_{K,X,Z}}\\
      \uS(Y,Z)\otimes\uV(K,\uS(X,Y)) \ar[rr]^(0.55){\delta} && \uV(K,\uS(X,Z))}
  \end{equation}
  commutes where $\delta=\pi^{-1}(\uV(K,\uS(X,-))_{Y,Z})$.
\end{lemma}
\begin{proof}
  Because of Lemma \ref{lem:cp2.4.14}, ${\underline\phi}_{K,X,Y}$ is $\V$-natural in $Y$ if and only if for every $Z\in\ob\S$, the following  diagram commutes.
  \begin{equation*}
    \resizebox{\textwidth}{!}{
    \xymatrix@C=10em{
      \uS(Y,Z)
      \ar[r]^(0.3){\uS(K\otimes X,-)}
      \ar[d]_{\uV(K,\uS(X,-))}
      &
      \uV(\uS(K\otimes X,Y),\uS(K\otimes X,Z))
      \ar[d]^{\uV(\uS(K\otimes X,Y),{\underline\phi}_{K,X,Z})}
      \\
      \uV(\uV(K,\uS(X,Y)),\uV(K,\uS(X,Z)))
      \ar[r]^{\uV({\underline\phi}_{K,X,Y},\uV(K,\uS(X,Z)))}
      &\uV(\uS(K\otimes X,Y),\uV(K,\uS(X,Z)))
      }}
  \end{equation*}
  But \eqref{eqn:cp2.2.5}, the commutativity of the diagram above is equivalent to the commutativity of \eqref{eqn:cp2.4.3}.
\end{proof}

\begin{lemma}
\label{lem:cp2.5.5}
  Let $\V$ be a closed monoidal category so that it is a $\V$-category by Proposition \ref{pro:cp2.3.5} and Proposition \ref{pro:cp2.3.4}.
  Let $\S$ be a $\V$-category.
  Let $K\in\ob\V$ and $X,Y,Z\in \ob\S$.
  Let
  \begin{equation*}
    \delta=\pi^{-1}(\uV(K,\uS(X,-))_{Y,Z}).
  \end{equation*}
  Then the following diagram commutes.
  \begin{equation}
  \label{eqn:cp2.4.5}
    \xymatrix{
         \uS(Y,Z)
         \otimes
         \uV(K,\uS(X,Y))
         \otimes
         K
      \ar[rr]^(0.55){\delta\otimes 1}
      \ar[d]_{1\otimes\epsilon}
      &&
         \uV(K,\uS(X,Z))
         \otimes
         K
      \ar[d]^\epsilon
      \\
         \uS(Y,Z)
         \otimes
         \uS(X,Y)
      \ar[rr]^\ub
      &&
      \uS(X,Z)
    }
  \end{equation}
\end{lemma}
\begin{proof}
  Using $\uV(K,\uS(X,-))=\uV(K,-)\cdot\uS(X,-)$ and \eqref{eqn:cp2.2.5}, we can decompose $\delta$  as follows.
  \begin{equation*}
    \resizebox{0.80\textwidth}{!}{
    \xymatrix@C=-2em{
      &
         \uV(\uS(X,Y),\uS(X,Z))
         \otimes
         \uV(K,\uS(X,Y))
      \ar[dr]^{\ub}
      &\\
         \uS(Y,Z)
         \otimes
         \uV(K,\uS(X,Y))
      \ar[rr]^\delta
      \ar[ur]^{\uS(X,-)\otimes 1}
      &&
      \uV(K,\uS(X,Z))
    }}
  \end{equation*}
  Then the top of the following diagram commutes.
  \begin{equation*}
    \resizebox{0.80\textwidth}{!}{
    \xymatrix@C=-3em{
      &
         \uV(\uS(X,Y),\uS(X,Z))
         \otimes
         \uV(K,\uS(X,Y))
         \otimes
         K
      \ar[dr]^(0.55){\ub\otimes 1}
      \ar[dd]|(0.5)\hole^(0.4){1\otimes\epsilon}
      &
      \\
         \uS(Y,Z)
         \otimes
         \uV(K,\uS(X,Y))
         \otimes
         K
      \ar[rr]^(0.4){\delta\otimes 1}
      \ar[ur]^(0.4){\uS(X,-)\otimes 1\otimes 1\quad}
      \ar[dd]_{1\otimes\epsilon}
      &&
         \uV(K,\uS(X,Z))
         \otimes
         K
      \ar[dd]^\epsilon
      \\
      &
          \uV(\uS(X,Y),\uS(X,Z))
          \otimes
          \uS(X,Y)
     \ar[dr]^\epsilon
      &
      \\
         \uS(Y,Z)
         \otimes
         \uS(X,Y)
      \ar[rr]^\ub
      \ar[ur]^(0.45){\uS(X,-)\otimes 1}
      &&
      \uS(X,Z)
    }}
  \end{equation*}
  The bottom face commutes by \eqref{eqn:cp2.2.5}. The right face commutes by \eqref{eqn:cp2.2.4}. The left face commutes for a trivial reason.
  Therefore \eqref{eqn:cp2.4.5} is a commutative diagram.
\end{proof}

\begin{proof}[Proof of Theorem \ref{thm:cp2.7.5}]
(1)
  Let $K\in\ob\V$ and $X,Y,Z\in \ob\S$.
  Consider the following diagram.
  \begin{equation*}
  \resizebox{\textwidth}{!}{
    \xymatrix@C=-2em{
        \uS(Y,Z)
        \otimes
        \uS(K\otimes X,Y)
        \otimes
        K
      \ar[rr]^{\ub\otimes 1}
      \ar[dr]^{1\otimes 1\otimes\alpha_{K,X}}
      \ar[dd]_{1\otimes{\underline\phi}_{K,X,Y}\otimes 1}
      &&
        \uS(K\otimes X,Z)
        \otimes
        K
      \ar[dd]_(0.3){{\underline\phi}_{K,X,Z}\otimes 1}
      \ar[dr]^{1\otimes\alpha_{K,X}}
      &
      \\
      &
         \uS(Y,Z)
         \otimes
         \uS(K\otimes X,Y)
         \otimes
         \uS(X,K\otimes X)
      \ar[rr]^(0.62){\ub\otimes 1}
      \ar[dd]_(0.3){1\otimes \ub}
      &&
        \uS(K\otimes X,Z)
        \otimes
        \uS(X,K\otimes X)
      \ar[dd]_{\ub}
      \\
         \uS(Y,Z)
         \otimes
         \uV(K,\uS(X,Y))
         \otimes
         K
      \ar[rr]^(0.7){\delta\otimes 1}
      \ar[dr]^{1\otimes \epsilon}
      &&
         \uV(K,\uS(X,Z))
         \otimes
         K
      \ar[dr]^{\epsilon}
      &
      \\
      &
         \uS(Y,Z)
         \otimes
         \uS(X,Y)
      \ar[rr]^{\ub}
      &&
      \uS(X,Z)
    }
  }
  \end{equation*}
  The top face commutes for a trivial reason. The left face and the right face commute because of \eqref{eqn:cp.1.1}. The front face commutes by the associativity of the composition of $\uS$. The bottom faces commutes by  Lemma \ref{lem:cp2.5.5}. Then using the face in the back, we obtain
  \[
    \pi^{-1}({\underline\phi}_{K,X,Z}\cdot\ub)=\pi^{-1}(\delta\cdot (1\otimes{\underline\phi}_{K,X,Y})).
  \]
  Therefore ${\underline\phi}_{K,X,Y}$ is $\V$-natural in $Y$ by Lemma \ref{lem:cp2.5.4}.

(2)
  Let $K\in\ob\V$ and $X,Y\in \ob\S$.
  Let $\eta_{K}^X={\underline\phi}_{K,X,K\otimes X}\cdot j_{K\otimes X}$.
  Consider the following commutative  diagram
  \begin{equation*}
  \resizebox{\textwidth}{!}{
    \xymatrix{
      \uS(K\otimes X,Y) \ar[r]^(0.45){r^{-1}} \ar[d]_1 &
      \uS(K\otimes X,Y)\otimes I \ar[r]^(0.35){j_{K\otimes X}} \ar[d]_1 &
      \uS(K\otimes X,Y)\otimes\uS(K\otimes X,K\otimes X) \ar[r]^(0.65)\ub \ar[d]_{1\otimes{\underline\phi}_{K,X,K\otimes X}} & \uS(K\otimes X,Y)\ar[d]_{{\underline\phi}_{K,X,Y}}\\
      \uS(K\otimes X,Y) \ar[r]^(0.45){r^{-1}} &
      \uS(K\otimes X,Y)\otimes I \ar[r]^(0.35){\eta_{K}^X} &
      \uS(K\otimes X,Y)\otimes\uV(K,\uS(X,K\otimes X)) \ar[r]^(0.65){\delta} & \uV(K,\uS(X,Y))}}
  \end{equation*}
  where $\delta=\pi^{-1}(\uV(K,\uS(X,-))_{K\otimes X, Y})$.
  The square on the right commutes by Lemma \ref{lem:cp2.5.4}.
  The top row is the identity because of \eqref{eqn:elno.2.3}. Thus the bottom row is ${\underline\phi}_{K,X,Y}$.
  Consider the following diagram.
  \begin{equation*}
  \resizebox{\textwidth}{!}{
    \xymatrix{
         \uS(K\otimes X,Y)
         \otimes
         I
         \otimes
         K
      \ar[r]^(0.38){\eta_{K}^X}
      &
         \uS(K\otimes X,Y)
         \otimes
         \uV(K,\uS(X,K\otimes X))
         \otimes
         K
      \ar[r]^(0.63){\delta\otimes 1}
      \ar[d]_{1\otimes\epsilon}
      &
         \uV(K,\uS(X,Y))
         \otimes
         K
      \ar[d]^\epsilon
      \\
         \uS(K\otimes X,Y)
         \otimes
         K
      \ar[u]^{1\otimes l^{-1}}
      \ar[r]^(0.42){1\otimes\varpi^{-1}(\eta_{K}^X)}
      &
         \uS(K\otimes X,Y)
         \otimes
         \uS(X,K\otimes X)
      \ar[r]^(0.63)\ub
      &
      \uS(X,Y)
    }}
  \end{equation*}
  By \eqref{eqn:elno.1.2} and the previous observation, the left upper corner is ${\underline\phi}_{K,X,Y}\otimes 1$.
  By the definition of the natural isomorphism $\varpi$ in \eqref{eqn:cp2.2.17}, the left square commutes.
  The right square also commutes by Lemma \ref{lem:cp2.5.5}.
  So if we let $\alpha_{K,X}=\varpi^{-1}({\underline\phi}_{K,X,K\otimes X}\cdot j_{K\otimes X})$, \eqref{eqn:cp.1.1} commutes.
\end{proof}

\begin{remark}
\label{rmk:cp2.5.3}
  Let $\V$ be a monoidal category.
  Let $\S$ be a $\V$-category.
  If $\V$ is closed and symmetric, then, using the result in chapter 1.10 of \cite{kelly-05}, one can show that
  $\S$ is tensored if and only if
  \begin{enumerate}
    \item there exists a $\V$-functor
        \begin{equation*}
        \label{eqn:cp2.5.2}
          -\otimes -:\V\otimes\S\rightarrow\S
        \end{equation*}
        whose object function is given by \eqref{eqn:cp2.7.1} and
    \item the morphism $\underline\phi_{K,X,Y}$ in \eqref{eqn:cp2.5.1} is $\V$-natural in $K,X,Y$.
  \end{enumerate}
\end{remark}

\section{Functoriality and Naturality of Cylinder}

Here we prove Proposition \ref{pro:esht.1.1} and Proposition \ref{pro:esht.1.2}. As we remarked at the end of the previous section,  stronger results hold if $\V$ is closed and symmetric.

\begin{lemma}
\label{lem:cp2.1.24}
  Let $\V$ be a closed monoidal category.
  Let $\S$ be a category with $\V$-structure.
  Let $K\in\ob\V$ and $X\in\ob\S$.
  Suppose that $\S$ has a $K$-cylinder of $X$.
  Then, for every $Y\in\ob\S$ and $f\in\S(K\otimes X, Y)$,
  \begin{equation}
  \label{eqn:cp2.3.1}
    \underline\phi_{K,X,Y}\cdot \varphi(f)=\varpi(\uS(X,f)\cdot\alpha_{K,X})
  \end{equation}
  holds where $\varphi$ is \eqref{eqn:cp2.4.15} and $\varpi$ is \eqref{eqn:cp2.2.17}.
\end{lemma}
\begin{proof}
  Consider the following commutative diagram.
  \[
    \resizebox{0.95\textwidth}{!}{
    \xymatrix{
      K
      \ar[d]_{\alpha_{K,X}}
      \ar[r]^(0.40){l^{-1}}
      &
      I \otimes K
      \ar[d]^{\alpha_{K,X}}
      \ar[r]^(0.40){\varphi(f)}
      &
      \uS({K\otimes X},Y)\otimes K
      \ar[r]^(0.48){\underline\phi_{K,X,Y}}
      \ar[d]^{\alpha_{K,X}}
      &
      \uV(K,\uS(X,Y))\otimes K
      \ar[d]^{\epsilon_{\uS(X,Y)}}
      \\
      K\otimes X
      \ar[r]^(0.40){l^{-1}}
      &
      I \otimes\uS(X,{K\otimes X})
      \ar[r]^(0.40){\varphi(f)}
      &
      \uS({K\otimes X},Y)\otimes\uS(X,{K\otimes X})
      \ar[r]^(0.66){\ub}
      &
      \uS(X,Y)
      }}
  \]
  The right commutes by \eqref{eqn:cp.1.1}.
  The top right corner of the diagram is $\varpi^{-1}(\underline\phi_{K,X,Y}\cdot \varphi(f))$. The bottom left corner is $\uS(X,f)\cdot\alpha_{K,X}$.
\end{proof}

\begin{definition}
\label{def:cp2.6.3}
  Let $\V$ be a closed monoidal category.
  Let $\S$ be a category with $\V$-structure.
  Let $K,L\in\ob\V$ and $X\in\ob\S$.
  Suppose that there exist a $K$-cylinder $K\otimes X$ of $X$ and a $L$-cylinder $L\otimes X$ of $X$.
  Then, for every $u\in\V(K,L)$,
  we denote by
  \begin{equation}
  \label{eqn:cp2.6.15}
    u\otimes X:K\otimes X\rightarrow L\otimes X
  \end{equation}
  the unique morphism
  that forces the following diagram commutes.
  \begin{equation}
  \label{eqn:cp2.1.1}
    \xymatrix{
      K \ar[r]^(0.3){\alpha_{K,X}} \ar[d]_u
        & \uS(X,K\otimes X) \ar[d]^{\uS(X,u\otimes X)}\\
      L \ar[r]_(0.3){\alpha_{L,X}} & \uS(X,L\otimes X)}
  \end{equation}
  Since $\underline\phi_{K,X,L\otimes X}$ is an isomorphism, $u\otimes X$ exists uniquely by Lemma \ref{lem:cp2.1.24}.
\end{definition}

\begin{definition}
\label{def:cp2.6.5}
  Let $\V$ be a closed monoidal category.
  Let $\S$ be a category with $\V$-structure.
  Let $K\in\ob\V$ and $X,Y\in\ob\S$.
  Suppose that there exist a $K$-cylinder $K\otimes X$ of $X$ and a $K$-cylinder $K\otimes Y$ of $Y$.
  Then, for every $v\in\S(X,Y)$,
  we denote by
  \begin{equation}
  \label{eqn:cp2.6.16}
    K\otimes v:K\otimes X\rightarrow K\otimes Y
  \end{equation}
  the unique morphism  that forces the following diagram commutes.
  \begin{equation}
  \label{eqn:cp2.1.2}
    \xymatrix{
      K \ar[rr]^(0.4){\alpha_{K,X}} \ar[d]_{\alpha_{K,Y}}
        && \uS(X,K\otimes X) \ar[d]^{\uS(X,K\otimes v)}\\
      \uS(Y,K\otimes Y) \ar[rr]_{\uS(v,K\otimes Y)}
        && \uS(X, K\otimes Y)}
  \end{equation}
  Since
  $\underline\phi_{K,X,K\otimes Y}$ is an isomorphism, $K\otimes v$ exists uniquely by Lemma \ref{lem:cp2.1.24}.
\end{definition}

\begin{proof}[Proof of Proposition \ref{pro:esht.1.1}]
  Since $\underline\phi_{K,X,X\otimes'K}$ is an isomorphism,
  there exists an unique morphism $f\in\S(K\otimes X,X\otimes'K)$
  making the diagram
  \begin{equation*}
    \xymatrix{
      K \ar[r]^(0.3){\alpha_{K,X}} \ar[d]_{\id_K}
        & \uS(X,K\otimes X) \ar[d]^{\uS(X,f)}\\
      K \ar[r]_(0.3){\alpha_{K,X}'} & \uS(X,X\otimes'K)}
  \end{equation*}
  commute by Lemma \ref{lem:cp2.1.24}.
\end{proof}

\begin{proof}[Proof of Proposition \ref{pro:esht.1.2}]
(1)
  Let $K,L\in\ob\V$ and let $X,Y\in\ob\S$.
  Let $u\in\V(K,L)$ and $v\in(X,Y)$.
  Using the uniqueness of of $u\otimes X$ and $K\otimes v$,
  one can easily prove that $-\otimes X$ and $K\otimes -$ are functors.
  To show that they form a bifunctor, we consider the following diagram.
  \[
    \xymatrix{
      K \ar[rr]^{\alpha_{K,X}} \ar[dr]^{\alpha_{K,Y}} \ar[dd]_u
        & & \uS(X,K\otimes X) \ar[dr]^{\;\;\uS(X,K\otimes v)}
        \ar[dd]|\hole^(0.7){\uS(X,u\otimes X)}& \\
      & \uS(Y,K\otimes Y) \ar[rr]^(0.35){\uS(v,K\otimes Y)}
        \ar[dd]^(0.7){\uS(Y,u\otimes Y)}
        & & \uS(X,K\otimes Y) \ar[dd]^{\uS(X,u\otimes Y)}\\
      L \ar[rr]^(0.2){\alpha_{L,X}}|(0.42)\hole \ar[dr]_{\alpha_{L,Y}}
        & & \uS(X,L\otimes X) \ar[dr]^{\;\;\uS(X,L\otimes v)} & \\
      & \uS(Y,L\otimes Y) \ar[rr]_{\uS(v,L\otimes Y)}& & \uS(X,L\otimes Y)}
  \]
  Since $\uS(-,-)$ is a bifunctor, the front face is a commutative diagram.
  The top and the bottom faces commute by the definition of $K\otimes v$ and $L\otimes v$ respectively. The left and the back faces commute by the definition $u\otimes Y$ and $u\otimes X$ respectively.
  So,
  \[
    \uS(X,(u\otimes Y)\cdot (K\otimes v))\cdot \alpha_{K,X}
    =
    \uS(X,(L\otimes v)\cdot (u\otimes X))\cdot \alpha_{K,X}.
  \]
  Then, by Lemma \ref{lem:cp2.1.24}
  \[
    \underline\phi_{K,X,Y\otimes L}\cdot\varphi((u\otimes Y)\cdot (K\otimes v))
    =
    \underline\phi_{K,X,Y\otimes L}\cdot\varphi((L\otimes v)\cdot (u\otimes X)).
  \]
  Hence,
  $(u\otimes Y)\cdot (K\otimes v)=(L\otimes v)\cdot (u\otimes X)$ holds.

(2)
  Let $K,L\in\ob\V$. Let $W,X,Y,Z\in\ob\S$.

(a)
  $\underline\phi_{-,X,Y}$: Let $u\in\V(K,L)$.
  Consider the following commutative diagram.
  \[
    \xymatrix{
      \uS(K\otimes X,Y)\otimes K
      \ar[rr]^(0.4){1\otimes\alpha_{K,X}}
      &&
      \uS(K\otimes X,Y)\otimes\uS(X,K\otimes X)
      \ar[dr]^{\ub}
      &
      \\
      \uS(L\otimes X, Y)\otimes K
      \ar[u]^{\uS(u\otimes X,Y)\otimes 1}
      \ar[d]_{1\otimes u}
      \ar[rr]^(0.4){1\otimes\alpha_{K,X}}
      && \uS(L\otimes X,Y)\otimes\uS(X,K\otimes X)
      \ar[u]^{\uS(u\otimes X,Y)\otimes 1}
      \ar[d]_{1\otimes\uS(X,u\otimes X)}
      & \uS(X,Y)
      \\
      \uS(L\otimes X,Y)\otimes L
      \ar[rr]_(0.4){1\otimes\alpha_{L,X}}
      && \uS(L\otimes X,Y)\otimes\uS(X,L\otimes X)
      \ar[ur]_{\ub}
      &}
  \]
  The bottom square commutes by \eqref{eqn:cp2.1.1}.
  The triangle commutes by \eqref{eqn:cp2.4.2}, \eqref{eqn:cp2.4.22} and~ \eqref{eqn:cp2.4.24}.
  From \eqref{eqn:cp.1.1}, we know $\pi(\underline\phi_{K,X,Y}\cdot \uS(u\otimes X,Y))$ is the upper half of the boundary of the  diagram and
  $\pi(\uV(u,\uS(X,Y))\cdot \underline\phi_{L,X,Y})$ is the lower half .
  So, the diagram
  \begin{equation*}
    \xymatrix{
      \uS(K\otimes X,Y) \ar[rr]^{\underline\phi_{K,X,Y}}
        && \uV(K,\uS(X,Y))\\
      \uS(L\otimes X,Y) \ar[rr]_{\underline\phi_{L,X,Y}} \ar[u]^{\uS(u\otimes X,Y)} &
        & \uV(L,\uS(X,Y)) \ar[u]_{\uV(u,\uS(X,Y))}}
  \end{equation*}
  commutes and $\underline\phi_{K,X,Y}$ is natural in $K$.

(b) $\underline\phi_{K,-,Y}$:
  Let $u\in\S(W,X)$.
  Consider the following commutative diagram.
  \[
    \xymatrix@C=4em{
      \uS(K\otimes W,Y)\otimes K
      \ar[r]^(0.4){1\otimes\alpha_{K,W}}
      & \uS(K\otimes W,Y)\otimes\uS(W,K\otimes W)
      \ar@/^3pc/[ddr]^(0.7){\ub}
      \\
      \uS(K\otimes X,Y)\otimes K
      \ar[r]^(0.4){1\otimes \alpha_{K,W}}
      \ar@/_3pc/[ddr]_(0.35){1\otimes\alpha_{K,X}}
      \ar[u]^{\uS(K\otimes u,Y)\otimes 1}
      & \uS(K\otimes X,Y)\otimes\uS(W,K\otimes W)
      \ar[u]^{\uS(K\otimes u,Y)\otimes 1}
      \ar[d]_{1\otimes\uS(W,K\otimes u)}
      &
      \\
      &\uS(K\otimes X,Y)\otimes\uS(W,K\otimes X)
      \ar[r]^(0.67)\ub
      &\uS(W,Y)
      \\
      &\uS(K\otimes X, Y)\otimes\uS(X,K\otimes X)
      \ar[r]_(0.67){\ub}
      \ar[u]^{1\otimes\uS(u,K\otimes X)}
      & \uS(X,Y)
      \ar[u]_{\uS(u,Y)}}
  \]
  The lower left commutes by \eqref{eqn:cp2.1.2}.
  The upper right commutes by \eqref{eqn:cp2.4.2}, \eqref{eqn:cp2.4.22} and~ \eqref{eqn:cp2.4.24}.
  The lower right commutes by \eqref{eqn:cp2.4.2} and~ \eqref{eqn:cp2.4.22}.
  From \eqref{eqn:cp.1.1}, we know $\pi(\underline\phi_{K,W,Y}\cdot\uS(K\otimes u,Y))$ is the upper half of the boundary of the  diagram and
  $\pi(\uV(K,\uS(u,Y))\cdot\underline\phi_{K,X,Y})$ is the lower half.
  So, the diagram
  \[
    \xymatrix{
      \uS(K\otimes W,Y) \ar[rr]^{\underline\phi_{K,W,Y}}
        && \uV(K,\uS(W,Y))\\
      \uS(K\otimes X,Y) \ar[rr]_{\underline\phi_{K,X,Y}} \ar[u]^{\uS(K\otimes u,Y)} &
        & \uV(K,\uS(X,Y)) \ar[u]_{\uV(K,\uS(u,Y))}}
  \]
  commutes and $\underline\phi_{K,X,Y}$ is natural in $X$.

(c) $\underline\phi_{K,X,-}$:
  Let $u\in\S(Y,Z)$.
  Consider the following commutative diagram.
  \[
    \xymatrix{
      \uS(K\otimes X, Y)\otimes K
      \ar[rr]^(0.43){1\otimes\alpha_{K,X}}
      \ar[d]_{\uS(K\otimes X,u)\otimes 1}
      &&\uS(K\otimes X,Y)\otimes\uS(X,K\otimes X)
      \ar[r]^(0.7){\ub}
      \ar[d]_{\uS(K\otimes X,u)\otimes 1}
      & \uS(X,Y)
      \ar[d]^{\uS(X,u)}
      \\
      \uS(K\otimes X,Z)\otimes K
      \ar[rr]_(0.43){1\otimes\alpha_{K,X}}
      && \uS(K\otimes X,Z)\otimes\uS(X,K\otimes X)
      \ar[r]_(0.7){\ub}
      &\uS(X,Z)}
  \]
  The right square commutes by \eqref{eqn:cp2.4.2} and~ \eqref{eqn:cp2.4.24}.
  From \eqref{eqn:cp.1.1}, we know $\pi(\uV(K,\uS(X,u))\cdot \underline\phi_{K,X,Y})$ is the upper right corner of the  diagram and
  $\pi(\underline\phi_{K,X,Z}\cdot \uS(K\otimes X,u))$ is the lower left.
  So, the diagram
  \[
    \xymatrix{
      \uS(K\otimes X,Y) \ar[rr]^{\underline\phi_{K,X,Y}} \ar[d]_{\uS(K\otimes X,u)}
        && \uV(K,\uS(X,Y)) \ar[d]^{\uV(K,\uS(X,u))}\\
      \uS(K\otimes X,Z) \ar[rr]_{\underline\phi_{K,X,Z}}&
        & \uV(K,\uS(X,Z))}
  \]
  commutes and $\underline\phi_{K,X,Y}$ is natural in $Y$.
\end{proof}

\section{Tensor-Closed $\V$-Module}
\label{sec:cp2.tcvm}

Here we prove Proposition \ref{pro:cp2.7.12} and Theorem \ref{thm:cp2.6.9}.
First, we generalize some properties of monoidal categories to that of $\V$-modules.

\begin{lemma}
\label{lem:cp2.7.1}
  Let $\V$ be a monoidal category.
  Let $\S$ be a $\V$-module.
  Then, for every $K\in\ob\V$ and $X\in\ob\S$, the following diagram commutes.
  \begin{equation}
  \label{eqn:cp2.7.6}
    \xymatrix{
      I\otimes (K\otimes X) \ar[rr]^a  \ar[dr]_l && (I\otimes K)\otimes X \ar[dl]^{l}  & \\
      & K\otimes X &}
  \end{equation}
\end{lemma}
\begin{proof}
  Like the diagram \eqref{eqn:cp2.2.6}, it also follows from \eqref{eqn:cp2.6.4} and \eqref{eqn:cp2.6.5}.
\end{proof}

\begin{lemma}
\label{eqn:cp2.7.2}
  Let $\V$ be a monoidal category.
  Let $\S$ be a tensor-closed $\V$-module.
  Then, for every $f\in\S(X,Y)$ and $Z\in\ob\S$, the following diagram commutes.
  \begin{equation}
  \label{eqn:cp2.7.9}
    \xymatrix{
      \uS(Y,Z)\otimes X \ar[r]^{1\otimes f} \ar[d]_{\uS(f,Z)\otimes 1} & \uS(Y,Z)\otimes Y \ar[d]^{\epsilon^Y_Z}\\
      \uS(X,Z)\otimes X \ar[r]^{\epsilon^X_Z} & Z}
  \end{equation}
\end{lemma}
\begin{proof}
  It follows from  $\phi_{K,X,Y}$ being natural in $K$ and $X$.
\end{proof}

Proposition \ref{pro:cp2.7.12} is a generalization of Proposition \ref{pro:cp2.3.5}. It is probably well-known. But we could not find a reference. So we give the proof.

\begin{proof}[Proof of Proposition \ref{pro:cp2.7.12}]
  Let $A,B,C,D\in\ob\S$.
  For the associativity law, we need to show that the following diagram commutes.
  \begin{equation}
  \label{eqn:cp2.3.21}
    \xymatrix@R=-0em{
    \uS(C,D)\otimes(\uS(B,C)\otimes\uS(A,B))\ar[r]^(0.56)\ub \ar[dd]_{\alpha^{-1}}^\cong & \uS(C,D)\otimes\uS(A,C) \ar[dr]^(0.6)\ub&\\
    && \uS(A,D)\\
    (\uS(C,D)\otimes\uS(B,C))\otimes\uS(A,B)\ar[r]^(0.56)\ub  & \uS(B,D)\otimes\uS(A,B) \ar[ur]_(0.6)\ub&}
  \end{equation}
  Consider the following diagram.
  \begin{equation*}
  \resizebox{\textwidth}{!}{
    \xymatrix@C=1.5em@R=1.2em{
      (\uS(C,D)\otimes(\uS(B,C)\otimes\uS(A,B)))\otimes A \ar[r]^(0.58)\ub \ar[d]_{a}^\cong & (\uS(C,D)\otimes\uS(A,C))\otimes A \ar[r]^(0.6)\ub \ar[d]_{a}^\cong & \uS(A,D)\otimes A \ar[r]^(0.66)\epsilon  & D \ar[d]^=
      \\
      \uS(C,D)\otimes((\uS(B,C)\otimes\uS(A,B))\otimes A) \ar[r]^(0.58)\ub \ar[d]_{a}^\cong & \uS(C,D)\otimes(\uS(A,C)\otimes A) \ar[r]^(0.6)\epsilon & \uS(C,D)\otimes C \ar[r]^(0.66)\epsilon \ar[d]^= & D \ar[d]^=
      \\
      \uS(C,D)\otimes(\uS(B,C)\otimes(\uS(A,B)\otimes A)) \ar[r]^(0.58)\epsilon \ar[d]_{{a}^{-1}}^\cong & \uS(C,D)\otimes(\uS(B,C)\otimes B) \ar[r]^(0.6)\epsilon \ar[d]_{{a}^{-1}}^\cong & \uS(C,D)\otimes C \ar[r]^(0.66)\epsilon & D \ar[d]^=
      \\
      (\uS(C,D)\otimes\uS(B,C))\otimes(\uS(A,B)\otimes A) \ar[r]^(0.58)\epsilon \ar[d]_= & (\uS(C,D)\otimes\uS(B,C))\otimes B \ar[r]^(0.6)\ub  & \uS(B,D)\otimes B \ar[r]^(0.66)\epsilon \ar[d]^= & D \ar[d]^=
      \\
      (\uS(C,D)\otimes\uS(B,C))\otimes(\uS(A,B)\otimes A) \ar[r]^(0.58)\ub \ar[d]_{{a}^{-1}}^\cong & \uS(B,D)\otimes(\uS(A,B)\otimes A) \ar[r]^(0.6)\epsilon \ar[d]_{{a}^{-1}}^\cong & \uS(B,D)\otimes B \ar[r]^(0.66)\epsilon & D \ar[d]^=
      \\
      ((\uS(C,D)\otimes\uS(B,C))\otimes\uS(A,B))\otimes A \ar[r]^(0.58)\ub & (\uS(B,D)\otimes\uS(A,B))\otimes A \ar[r]^(0.6)\ub & \uS(A,D)\otimes A \ar[r]^(0.66)\epsilon & D\\}}
  \end{equation*}
  All the inner squares commute by the definition \eqref{eqn:cp2.7.11} of $\ub$ or trivial reasons. Therefore it is a commutative diagram.
  By \eqref{eqn:cp2.6.4}, the first column is $a^{-1}\otimes 1$. Thus $\pi^{-1}(\eqref{eqn:cp2.3.21})$ is a commutative diagram.

  Next, we prove (5) in Definition \ref{def:cp2.4.5}.
  Let $f\in\S(A,B)$. Consider the following commutative diagram.
  \begin{equation*}
    \xymatrix@C=1.5em@R=1.7em{
      \uS(B,C)\otimes A \ar[r]^{r^{-1}}\ar[dr]^1 &\uS(B,C)\otimes\i\otimes A \ar[r]^(0.45){\varphi(f)}\ar[d]^l &\uS(B,C)\otimes\uS(A,B)\otimes A\ar[r]^(0.6)\ub\ar[d]^\epsilon & \uS(A,C)\otimes A \ar[d]^\epsilon\\
      & \uS(B,C)\otimes A\ar[r]^{f} & \uS(B,C)\otimes B\ar[r]^\epsilon& C}
  \end{equation*}
  The triangle commutes by \eqref{eqn:cp2.6.5} and two squares commute by the definitions.
  The upper-right corner of the diagram is $\phi^{-1}(\eqref{eqn:cp2.4.22})$.
  By \eqref{eqn:cp2.7.9},
  the bottom row is
  \begin{equation*}
    \uS(B,C)\otimes A\xrightarrow{\uS(f,C)\otimes 1} \uS(A,C)\otimes A\xrightarrow{\epsilon} C,
  \end{equation*}
  which is $\phi^{-1}(\uS(f,C))$. Therefore (5) holds.

  Finally, we prove (6) in Definition \ref{def:cp2.4.5}. Let $g\in\S(B,C)$. The upper-right corner of the following commutative diagram is $\phi^{-1}(\eqref{eqn:cp2.4.24})$.
  \begin{equation*}
    \xymatrix@C=1.5em@R=1.7em{
      \uS(A,B)\otimes A \ar[r]^{l^{-1}}\ar[d]^\epsilon &\i\otimes\uS(A,B)\otimes A \ar[r]^(0.45){\varphi(g)}\ar[d]^\epsilon &\uS(B,C)\otimes\uS(A,B)\otimes A\ar[r]^(0.6)\ub\ar[d]^\epsilon & \uS(A,C)\otimes A \ar[d]^\epsilon\\
      B\ar[r]^{l^{-1}} & \i\otimes B\ar[r]^{\varphi(g)} & \uS(B,C)\otimes B\ar[r]^\epsilon& C}
  \end{equation*}
  By the definition of $\varphi$, the bottom row is $g$.
  Then by the naturality of $\phi_{X,Y,Z}$ in $Z$, the bottom-left corner is $\phi^{-1}(\uS(A,g))$. Therefore (6) holds.
\end{proof}

The following lemma is a generalization of Lemma \ref{lem:cp2.4.14}.

\begin{lemma}
\label{lem:cp2.8.15}
  Let $\V$ be a monoidal category.
  Let $\A$ be a $\V$-category.
  Let $\S$ be a tensor-closed $\V$-module so that $\S$ is a $\V$-category by Proposition \ref{pro:cp2.7.12} and Proposition \ref{pro:cp2.3.4}.
  Let $S,T:\A\rightarrow \S$ be $\V$-functors.
  Then $\{\alpha_{A}:I\rightarrow\uS(SA,TA)\mid A\in\ob\A\}$ is a $\V$-natural transformation from $S$ to $T$ if and only if
  for every $A,B\in\ob\A$, the following diagram commutes.
  \begin{equation}
    \xymatrix@C=6em{
      \uA(A,B)
      \ar[r]^{S_{A,B}}
      \ar[d]_{T_{A,B}}
      & \uS(SA,SB)
      \ar[d]^{\uS(1,\varphi^{-1}(\alpha_B))}
      \\
      \uS(TA,TB)
      \ar[r]^{\uS(\varphi^{-1}(\alpha_A),1)}
      & \uS(SA,TB)}
  \end{equation}
\end{lemma}
\begin{proof}
  It also follows from \eqref{eqn:cp2.4.22} and \eqref{eqn:cp2.4.24}.
\end{proof}

\begin{definition}
  Let $\V$ be a closed monoidal category.
  Let $\S$ be a tensor-closed $\V$-module.
  Let $X\in\ob\S$.
  For every $K,L\in\ob\V$, we define a morphism
  \begin{equation}
  \label{eqn:cp2.7.15}
    (-\otimes X)_{K,L}:\uV(K,L)\rightarrow\uS(K\otimes X,L\otimes X)
  \end{equation}
  of $\V$
  by
  \begin{equation}
  \label{eqn:cp2.7.16}
    (-\otimes X)_{K,L}=\phi((\epsilon_L\otimes 1)\cdot a).
  \end{equation}
\end{definition}

The following four lemmas are well-known at least for closed monoidal categories.

\begin{lemma}
\label{lem:cp2.7.5}
  Let $\V$ be a closed monoidal category.
  Let $\S$ be a tensor-closed $\V$-module.
  Then, for every $K,L,M\in\ob\V$ and $X\in\ob\S$, the following diagram commutes.
  \begin{equation}
  \label{eqn:cp2.7.19}
    \xymatrix{
      \uV(L,M)\otimes\uV(K,L)
      \ar[rr]^\ub
      \ar[d]_{1\otimes (-\otimes X)_{K,L}}
      &&
      \uV(K,M) \ar[dd]^{(-\otimes X)_{K,M}}
      \\
      \uV(L,M)\otimes\uS(K\otimes X,L\otimes X)
      \ar[d]_{(-\otimes X)_{L,M}\otimes 1}&&
      \\
      \uS(L\otimes X,M\otimes X)\otimes\uS(K\otimes X,L\otimes X) \ar[rr]^(0.6)\ub
        && \uS(K\otimes X,M\otimes X)}
  \end{equation}
\end{lemma}
\begin{proof}
  From \eqref{eqn:cp2.2.4}, \eqref{eqn:cp2.7.11} and \eqref{eqn:cp2.7.16}, all the inner diagrams of the following diagram commute.
  \begin{equation*}
  \resizebox{0.95\textwidth}{!}{
    \xymatrix@C=-3em@R=1.2em{
      \uV(L,M)\otimes\uV(K,L)\otimes K\otimes X \ar[rrr]^\ub \ar[dr]^\epsilon \ar[dd]_{(-\otimes X)_{K,L}}
        &&& \uV(K,M)\otimes K\otimes X \ar[dddd]^{(-\otimes X)_{K,M}} \ar[ldd]_\epsilon\\
      & \uV(L,M)\otimes L\otimes X \ar[dd]_{(-\otimes X)_{L,M}} \ar[dr]^\epsilon & &\\
      \uV(L,M)\otimes\uS(K\otimes X,L\otimes X)\otimes K\otimes X \ar[dd]_{(-\otimes X)_{L,M}} \ar[ur]^\epsilon
      && M\otimes X&\\
      &\uS(L\otimes X,M\otimes X)\otimes L\otimes X \ar[ur]_(0.65)\epsilon
      &&\\
      \uS(L\otimes X,M\otimes X)\otimes\uS(K\otimes X,L\otimes X)\otimes K\otimes X \ar[ur]^\epsilon \ar[rrr]^(0.58)\ub
        &&& \uS(K\otimes X,M\otimes X)\otimes K\otimes X \ar[luu]_\epsilon}}
  \end{equation*}
  But taking $\phi$ on this diagram, we get \eqref{eqn:cp2.7.19}.
\end{proof}

\begin{lemma}
\label{lem:cp2.7.14}
  Let $\V$ be a closed monoidal category so that $\V$ is a $\V$-category by Proposition \ref{pro:cp2.3.5} and Proposition \ref{pro:cp2.3.4}.
  Let $\S$ be a tensor-closed $\V$-module so that $\S$ is a $\V$-category by Proposition \ref{pro:cp2.7.12} and Proposition \ref{pro:cp2.3.4}.
  Let $X\in\ob\S$.
  Then,
  morphisms in \eqref{eqn:cp2.7.15} form a $\V$-functor
  \begin{equation}
  \label{eqn:cp2.7.14}
    -\otimes X:\V\rightarrow\S.
  \end{equation}
\end{lemma}

\begin{proof}
  The commutativity of \eqref{eqn:elno.2.6} holds for $-\otimes X$ by Lemma \ref{lem:cp2.7.5}.
  Next, since $\pi(l_K)=j_K$ by \eqref{eqn:cp2.4.8},
  the following diagram commutes.
  \begin{equation*}
    \xymatrix{
      \uV(K,K)\otimes (K\otimes X) \ar[r]^a & (\uV(K,K)\otimes K)\otimes X \ar[r]^(0.65)\epsilon & K\otimes X\\
      I\otimes (K\otimes X) \ar[r]^a \ar[u]^{j_K} & (I\otimes K)\otimes X \ar[ur]_{l} \ar[u]^{j_K} & }
  \end{equation*}
  Then from \eqref{eqn:cp2.7.6} and \eqref{eqn:cp2.4.8}, $(-\otimes X)_{K,K}\cdot j_K=\phi(l_{K\otimes X})=j_{K\otimes X}$.
  Thus the commutativity of \eqref{eqn:elno.2.7} also holds for $-\otimes X$.
\end{proof}

\begin{remark}
  Let $u\in\uV(K,L)$.
  By the definition of $\varpi$,  the following
  diagram commutes.
  \begin{equation*}
    \xymatrix{
      \uV(K,L)\otimes (K\otimes X) \ar[r]^a & (\uV(K,L)\otimes K)\otimes X \ar[r]^(0.65)\epsilon & L\otimes X\\
      I\otimes (K\otimes X) \ar[r]^a \ar[u]^{\varpi(u)} & (I\otimes K)\otimes X \ar[r]^(0.56){l\otimes 1} \ar[u]^{\varpi(u)} & K\otimes X \ar[u]_{u\otimes 1}}
  \end{equation*}
  The upper left corner is $\phi^{-1}((-\otimes X)_{K,L}\cdot \varpi(u))$.
  The lower right corner is $(u\otimes X)\cdot l_{K\otimes X}$ by \eqref{eqn:cp2.7.6}.
  Therefore $(-\otimes X)_{K,L}\cdot\varpi(u)=\varphi(u\otimes X)$  holds.
  So, we may say that
  the $\V$-functor in \eqref{eqn:cp2.7.14} extends
  the ordinary functor
  $-\otimes X:\V\rightarrow\S$ in \eqref{eqn:cp2.6.0}.
\end{remark}

\begin{lemma}
\label{lem:cp2.7.15}
  Let $\V$ be a closed monoidal category so that $\V$ is a $\V$-category by Proposition \ref{pro:cp2.3.5} and Proposition \ref{pro:cp2.3.4}.
  Let $\S$ be a tensor-closed $\V$-module so that $\S$ is a $\V$-category by Proposition \ref{pro:cp2.7.12} and Proposition \ref{pro:cp2.3.4}.
  Let $K,L\in\ob\V$ and $X\in\ob\S$.
  Then $(-\otimes X)_{K,L}$ in \eqref{eqn:cp2.7.15} is $\V$-natural in $L$.
\end{lemma}
\begin{proof}
  Because of Lemma \ref{lem:cp2.4.14}, it is enough to show that the following diagram commutes: let $M\in\ob\V$.
  \begin{equation*}
  \resizebox{0.95\textwidth}{!}{
    \xymatrix@C=6.5em{
      \uV(L,M) \ar[r]^(0.45){\uV(K,-)_{L,M}} \ar[d]_{(-\otimes X)_{L,M}} & \uV(\uV(K,L),\uV(K,M)) \ar[dd]^{\uV(1,(-\otimes X)_{K,M})} \\
      \uS(L\otimes X,M\otimes X)\ar[d]_{\uS(K\otimes X,-)_{L\otimes X,M\otimes X}} &\\
      \uV(\uS(K\otimes X,L\otimes X),\uS(K\otimes X,M\otimes X)) \ar[r]^(0.54){\uV((-\otimes X)_{K,L},1)}  & \uV(\uV(K,L),\uS(K\otimes X,M\otimes X))}}
  \end{equation*}
  But the above diagram is $\pi(\eqref{eqn:cp2.7.19})$.
\end{proof}

\begin{lemma}
\label{lem:cp2.7.16}
  Let $\V$ be a closed monoidal category so that $\V$ is a $\V$-category by Proposition \ref{pro:cp2.3.5} and Proposition \ref{pro:cp2.3.4}.
  Let $\S$ be a tensor-closed $\V$-module so that $\S$ is a $\V$-category by Proposition \ref{pro:cp2.7.12} and Proposition \ref{pro:cp2.3.4}.
  Let $X,Y\in\ob\S$.
  Then $\epsilon_Y:\uS(X,Y)\otimes X\rightarrow Y$ in \eqref{eqn:cp2.7.13} is $\V$-natural in $Y$.
\end{lemma}
\begin{proof}
  Because of Lemma \ref{lem:cp2.8.15}, it is enough to show that the following diagram commutes: let $Z\in\ob\S$.
  \begin{equation}
  \label{eqn:cp2.7.22}
    \xymatrix{
      \uS(Y,Z) \ar[rrr]^{\uS(\epsilon_Y,1)} \ar[d]_{\uS(X,-)_{Y,Z}}
      &&& \uS(\uS(X,Y)\otimes X,Z) \\
      \uV(\uS(X,Y),\uS(X,Z)) \ar[rrr]^(0.44){(-\otimes X)_{\uS(X,Y),\uS(X,Z)}}
      &&&
      \uS(\uS(X,Y)\otimes X,\uS(X,Z)\otimes X) \ar[u]_(0.55){\uS(1,\epsilon_Z)}}
  \end{equation}
  Consider the following diagram.
  \begin{equation*}
  \resizebox{0.95\textwidth}{!}{
    \xymatrix@C=0em@R=1.2em{
      \uS(Y,Z)\otimes\uS(X,Y)\otimes X
      \ar[rr]^{\uS(\epsilon_Y,1)}
      \ar[dd]_{\uS(X,-)_{Y,Z}}
      \ar[dr]^{1\otimes\epsilon}
      \ar[dddr]^{\ub\otimes 1}
      &
      & \uS(\uS(X,Y)\otimes X,Z)\otimes\uS(X,Y)\otimes X
      \ar[ddddd]^{\epsilon}
      \\
      & \uS(Y,Z)\otimes Y
      \ar[ddddr]^\epsilon
      &
      \\
      \uV(\uS(X,Y),\uS(X,Z))\otimes\uS(X,Y)\otimes X
      \ar[dd]_{(-\otimes X)_{\uS(X,Y),\uS(X,Z)}}
      \ar[dr]_{\epsilon\otimes 1}
      &
      &
      \\
      & \uS(X,Z)\otimes  X
      \ar[ddr]_\epsilon
      &
      \\
      \uS(\uS(X,Y)\otimes X,\uS(X,Z)\otimes X)\otimes\uS(X,Y)\otimes X
      \ar[d]_{\uS(1,\epsilon_Z)}
      \ar[ur]^\epsilon
      &&
      \\
      \uS(\uS(X,Y)\otimes X,Z)\otimes\uS(X,Y)\otimes X
      \ar[rr]^{\epsilon}
      && Z
      }}
  \end{equation*}
  From the upper right to the lower left, all the inner diagrams commute because of \eqref{eqn:cp2.7.9}, \eqref{eqn:cp2.7.11}, \eqref{eqn:cp2.2.5}, \eqref{eqn:cp2.7.16} and the naturality of $\epsilon^{\uS(X,Y)\otimes X}$.
  Therefore $\phi^{-1}(\eqref{eqn:cp2.7.22})$ commutes.
\end{proof}

\begin{definition}
\label{def:cp2.7.10}
  Let $\V$ be a closed monoidal category.
  Let $\S$ be a tensor-closed $\V$-module.
  For every $K\in\ob\V$ and $X,Y\in\ob\S$,
  we define a morphism $\underline\phi_{K,X,Y}$ of $\V$
  as the unique morphism that makes the following diagram commute
  \begin{equation}
  \label{eqn:cp2.7.4}
    \xymatrix{
      \S((L\otimes K)\otimes X,Y)\ar[r]^{\phi} &
      \V(L\otimes K,\uS(X,Y))\ar[r]^{\pi} &
      \V(L,\uV(K,\uS(X,Y)))\\
      \S(L\otimes (K\otimes X),Y)\ar[u]^{a^*} \ar[rr]^{\phi} && \V(L,\uS(K\otimes X,Y))\ar[u]_{\V(L,\underline{\phi}_{K,X,Y})}.}
  \end{equation}
  where $L\in\ob\V$. We note that if $\S$ is $\V$, the diagram \eqref{eqn:cp2.7.4} is reduced to \eqref{eqn:cp2.2.1}.
\end{definition}

\begin{lemma}[cf.~1.8.(j) in \cite{kelly-05}]
\label{lem:cp2.7.13}
  Let $\V$ be a closed monoidal category.
  Let $\S$ be a tensor-closed $\V$-module.
  Then for every $K\in\ob\V$ and $X,Y\in\ob\S$,
  $\underline\phi_{K,X,Y}$ in Definition \ref{def:cp2.7.10} is the inverse of the composition
  \begin{equation*}
    \uS(K\otimes X,Y)\xleftarrow{\uS(1,\epsilon_Y)} \uS(K\otimes X,\uS(X,Y)\otimes X) \xleftarrow{(-\otimes X)_{K,\uS(X,Y)}} \uV(K,\uS(X,Y)).
  \end{equation*}
\end{lemma}
\begin{proof}
  Consider the following commutative diagram.
  \begin{equation}
  \label{eqn:cp2.7.24}
    \xymatrix{
      \V(L,\uV(K,\uS(X,Y)))
      \ar[d]_{(-\otimes X)_{K,\uS(X,Y)}}
      &
      \\
      \V(L,\uS(K\otimes X,\uS(X,Y)\otimes X))
      \ar[r]_\cong^\pi
      \ar@{.>}[d]_\epsilon
      &
      \S(L\otimes (K\otimes X),\uS(X,Y)\otimes X)
      \ar[d]^\epsilon
      \\
      \V(L,\uS(K\otimes X,Y))
      \ar@{.>}[r]_\cong^\pi
      &
      \S(L\otimes (K\otimes X),Y)
    }
  \end{equation}
  Let $f\in\V(L,\uV(K,\uS(X,Y)))$.
  From \eqref{eqn:cp2.7.16}, we know
  the composition of the straight arrows in the following commutative diagram is the image of $f$ in $\S(L\otimes (K\otimes X),Y)$ computed along the straight arrows in \eqref{eqn:cp2.7.24}.
  \begin{equation}
  \label{eqn:cp2.7.25}
    \xymatrix@C=4em{
      L\otimes(K\otimes X)
      \ar[r]^(0.37){f\otimes 1}
      \ar@{.>}[d]_{a^{-1}}
      & \uV(K,\uS(X,Y))\otimes(K\otimes X)
      \ar[d]^{a^{-1}}
      &
      \\
      (L\otimes K)\otimes X
      \ar@{.>}[r]^(0.37){(f\otimes 1)\otimes 1}
       & (\uV(K,\uS(X,Y))\otimes K)\otimes X
      \ar[d]^{\epsilon\otimes 1}
      &
      \\
      & \uS(X,Y)\otimes X
      \ar[r]^\epsilon
      & Y}
  \end{equation}
  But $\epsilon\cdot (\epsilon\otimes 1)\cdot ((f\otimes 1)\otimes 1)$ in \eqref{eqn:cp2.7.25} is the image of $f$ in $\S((L\otimes K)\otimes X,Y)$
  when we trace it along \eqref{eqn:cp2.7.4}.
  So the claim follows.
\end{proof}

\begin{proof}[Proof of Theorem \ref{thm:cp2.6.9}.(1)]
  The object $K\otimes X$ in \eqref{eqn:cp2.7.1} is given by the functor in \eqref{eqn:cp2.6.0}.
  If we define $\underline\phi_{K,X,Y}$ in \eqref{eqn:cp2.5.1} as the morphism in Definition \ref{def:cp2.7.10},
  by Lemma \ref{lem:cp2.7.13},
  Lemma \ref{lem:cp2.7.15} and Lemma \ref{lem:cp2.7.16}, $\underline\phi_{K,X,Y}$ is $\V$-natural in $Y$ and $\S$ is tensored.
\end{proof}

Before we prove Theorem \ref{thm:cp2.6.9}.(2), we prove two consequences of Theorem \ref{thm:cp2.6.9}.(1).

\begin{proposition}
\label{pro:cp2.7.1}
  Let $\V$ be a closed monoidal category. If $\V$ is equipped with a $\V$-structure in Proposition \ref{pro:cp2.3.5}, then $\V$ has cylinder:
  for every $K,L\in\ob\V$, the $K$-cylinder of $L$ is
  \begin{equation}
  \label{eqn:cp2.7.23}
    (K\otimes L, \pi_{K,L,K\otimes L}(1_{K\otimes L}), \{\underline\pi_{K,L,M}\}_{M\in\ob\V})
  \end{equation}
  where $\underline\pi_{K,L,M}$ is \eqref{eqn:cp2.2.3}.
\end{proposition}
\begin{proof}
  By Lemma \ref{lem:cp2.2.10} and \eqref{eqn:cp2.4.8}, $\underline\pi_{K,L,K\otimes L}\cdot j_{K\otimes L}=
  \varpi(\pi_{K,L,K\otimes L}(1_{K\otimes L}))$
  holds. If $\S$ is $\V$, the diagram \eqref{eqn:cp2.7.4} defining $\underline\phi_{K,X,Y}$ is reduced to \eqref{eqn:cp2.2.1} defining $\underline\pi_{K,L,M}$. Therefore \eqref{eqn:cp2.1.16} in Theorem \ref{thm:cp2.7.5} reduces to \eqref{eqn:cp2.7.23}.
\end{proof}

\begin{lemma}
  Let $\V$ be a closed monoidal category equipped with a $\V$-structure in Proposition \ref{pro:cp2.3.5}.
  Then for every $K,L,M\in\ob\V$,
  \begin{equation}
  \label{eqn:cp2.7.5}
    \pi_{-,K\otimes L,-}^{-1}(\underline\pi_{K,L,M}^{-1})=\epsilon^L_M\cdot(\epsilon^K_M\otimes 1)\cdot a
  \end{equation}
  holds where $\underline\pi_{K,L,M}$ is \eqref{eqn:cp2.2.3}.
\end{lemma}
\begin{proof}
  We know from Lemma \ref{lem:cp2.7.13} that
  $\pi_{-,K\otimes L,-}^{-1}(\underline\pi_{K,L,M}^{-1})$ is the upper right corner of the following diagram.
  \begin{equation*}
    \resizebox{1.1\textwidth}{!}{
    \xymatrix@C=6em{
    \uV(K,\uV(L,M))\otimes K\otimes L \ar[r]^(0.45){(-\otimes L)_{K,\uV(L,M)}}
    \ar[dr]_{\epsilon^K_M}
    &
    \uV(K\otimes L,\uV(L,M)\otimes L)\otimes K\otimes L \ar[r]^(0.56){\uV(1,\epsilon_M^L)}
    \ar[d]^{\epsilon^{K\otimes L}_{\uV(L,M)\otimes L}}
    &
    \uV(K\otimes L,M)\otimes K\otimes L
    \ar[d]^{\epsilon^{K\otimes L}_M}\\
    &
    \uV(L,M)\otimes L
    \ar[r]_{\epsilon^L_M}
    &
    M}}
  \end{equation*}
  Because of \eqref{eqn:cp2.7.16}, the triangle commutes.
  Hence \eqref{eqn:cp2.7.5} holds.
\end{proof}

\begin{proof}[Proof of Theorem \ref{thm:cp2.6.9}.(2)]
(A)
  First, we prove that if $\S$ is a category with $\V$-structure and cylinder, then $\S$ is a tensor-closed $\V$-module.

  The functor $-\otimes -$ in \eqref{eqn:cp2.6.0} is provided by Proposition \ref{pro:esht.1.2}.(1).
  The functor $\uS(-,-)$ in \eqref{eqn:cp2.6.6} is given by the functor in \eqref{eqn:cp2.4.17}.

  The natural isomorphism $\phi_{K,X,Y}$ in \eqref{eqn:cp2.6.7} is defined by the following commutative diagram
  \begin{equation}
  \label{eqn:cp2.7.32}
    \xymatrix{
      \S(K\otimes X,Y)
      \ar[r]^(.45)\varphi_(.45)\cong
      \ar[d]_{\phi_{K,X,Y}}
      &
      \V(I,\uS(K\otimes X,Y))
      \ar[d]^{\V(I,\underline\phi_{K,X,Y})}
      \\
      \V(K,\uS(X,Y))
      \ar[r]_(.45)\varpi^(.45)\cong
      &
      \V(I,\uV(K,\uS(X,Y)))}
  \end{equation}
  where $\underline\phi_{K,X,Y}$ is \eqref{eqn:cp2.4.1}.
  Then by Proposition \ref{pro:esht.1.2}.(2), $\phi_{K,X,Y}$ is natural in $K,X,Y$.
  We observe that for every $f\in\S(K\otimes X,Y)$, $\phi_{K,X,Y}(f)$ is the composition
  \begin{equation}
  \label{eqn:cp2.7.3}
    K\xrightarrow{l^{-1}} I\otimes K\xrightarrow{\varphi(f)\otimes\alpha_{K,X}}
    \uS(K\otimes X, Y)\otimes\uS(X,K\otimes X) \xrightarrow{\ub} \uS(X,Y)
  \end{equation}
  by \eqref {eqn:cp.1.1} and \eqref{eqn:cp2.2.17}.
  In particular,
  \begin{equation}
  \label{eqn:cp2.7.8}
    \phi_{K,X,K\otimes X}(1_{K\otimes X})=\alpha_{K,X}.
  \end{equation}

  Next, we define the natural transformation $a_{K,L,X}$ in \eqref{eqn:cp2.6.1}
  by the commutativity of the diagram \eqref{eqn:cp2.1.28} and Yoneda Lemma.
  \begin{equation}
  \label{eqn:cp2.1.28}
    \xymatrix{
      \S((K\otimes L)\otimes X,Y)\ar[r]^{\phi} &
      \V(K\otimes L,\uS(X,Y))\ar[r]^{\pi} &
      \V(K,\uV(L,\uS(X,Y)))\\
      \S(K\otimes (L\otimes X),Y)\ar[u]^{a_{K,L,X}^*} \ar[rr]_{\phi} && \V(K,\uS(L\otimes X,Y))\ar[u]_{\V(K,\underline{\phi})}}
  \end{equation}
  where $K,L\in\ob\V$, $X,Y\in\ob \S$.
  Using \eqref{eqn:cp2.7.3}, we know
  \begin{equation*}
    \phi_{K\otimes L, X, K\otimes(L\otimes X)}(a_{K,L,X})=\ub\cdot (\varphi(a_{K,L,X})\otimes\alpha_{K\otimes L,X})\cdot l^{-1}_{K\otimes L}.
  \end{equation*}
  On the other hand, by \eqref{eqn:cp2.7.8},
  \begin{align*}
  \label{eqn:cp2.7.17}
    &(\pi^{-1}_{K,L,\uS(X,K\otimes (L\otimes X))}\cdot\V(K,\underline\phi_{L,X,K\otimes (L\otimes X)})\cdot\phi_{K,L\otimes X, K\otimes(L\otimes X)})(1_{K\otimes (L\otimes X)})\\
    =&\pi^{-1}_{K,L,\uS(X,K\otimes (L\otimes X))}(\underline\phi_{L,X,K\otimes (L\otimes X)}\cdot\alpha_{K,L\otimes X})\\
    =&\ub\cdot (\alpha_{K,L\otimes X}\otimes\alpha_{L,X}).
  \end{align*}
  So, by the commutativity of \eqref{eqn:cp2.1.28},
  \begin{equation}
  \label{eqn:cp2.7.7}
    \ub\cdot (\alpha_{K,L\otimes X}\otimes\alpha_{L,X})=\ub\cdot (\varphi(a_{K,L,X})\otimes\alpha_{K\otimes L,X})\cdot l^{-1}_{K\otimes L}.
  \end{equation}
  Consider the following diagram.
  \begin{equation}
  \label{eqn:cp2.7.21}
    \xymatrix@C=2em{
      \uS((K\otimes L)\otimes X,Y)\ar[r]^{\underline\phi} &
      \uV(K\otimes L,\uS(X,Y))\ar[r]^{\underline\pi} &
      \uV(K,\uV(L,\uS(X,Y)))\\
      \uS(K\otimes (L\otimes X),Y)\ar[u]^{\uS(a_{K,L,X},Y)} \ar[rr]_{\underline\phi} && \uV(K,\uS(L\otimes X,Y))\ar[u]_{\V(K,\underline{\phi})}}
  \end{equation}
  Using \eqref{eqn:cp.1.1} and \eqref{eqn:cp2.4.24},
  we know $\pi^{-1}_{-,K\otimes L,-}(\underline\phi\cdot\uS(a_{K,L,X},Y))$ can be represented by the straight arrows in the following commutative diagram: for simplicity, we let $M=\uS(K\otimes (L\otimes X),Y)$.
  \begin{equation*}
    \resizebox{1.15\textwidth}{!}{
    \xymatrix@C=4em{
      M\otimes (K\otimes L) \ar[d]_{l^{-1}}&\\
      M\otimes I\otimes (K\otimes L)
      \ar[d]_{\varphi(a_{K,L,X})}&\\
      M\otimes \uS((K\otimes L)\otimes X, K\otimes(L\otimes X))\otimes (K\otimes L)
      \ar@{..>}[d]_\ub\ar[r]^(0.45){\alpha_{K\otimes L ,X}}
      &M\otimes\uS((K\otimes L)\otimes X, K\otimes(L\otimes X))\otimes\uS(X,(K\otimes L)\otimes X)
      \ar[d]^\ub\\
      \uS((K\otimes L)\otimes X,Y)
      \ar@{..>}[r]^{\alpha_{K\otimes L ,X}}\otimes (K\otimes L)\ar[d]_{\underline\phi}
      &\uS((K\otimes L)\otimes X,Y)\otimes \uS(X,(K\otimes L)\otimes X)
      \ar[d]^\ub
      \\
     \uV(K\otimes L,\uS(X,Y))\otimes(K\otimes L)
      \ar@{..>}[r]_\epsilon
      &\uS(X,Y)}}
  \end{equation*}
  On the other hand, using \eqref{eqn:cp2.7.5},
  $\pi^{-1}_{-,K\otimes L,-}(\underline\pi^{-1}\cdot\uV(K,\underline\phi)\cdot\underline\phi)$ can be represented by the straight arrows in the following commutative diagram.
  \begin{equation*}
    \resizebox{0.95\textwidth}{!}{
    \xymatrix@C=3em{
      \uS(K\otimes (L\otimes X),Y)\otimes K\otimes L
      \ar@{..>}[d]_{\underline\phi}\ar[r]^(0.38){\alpha_{K,L\otimes X}}
      &
      \uS(K\otimes (L\otimes X),Y)\otimes\uS(L\otimes X,K\otimes (L\otimes X))\otimes L
      \ar[d]^\ub\\
      \uV(K,\uS(L\otimes X,Y))\otimes K\otimes L
      \ar@{..>}[d]_{\uV(K,\underline\phi)}\ar@{..>}[r]^{\epsilon^K}
      &
      \uS(L\otimes X,Y)\otimes L
      \ar[d]^{\underline\phi}
      \\
      \uV(K,\uV(L,\uS(X,Y)))\otimes K\otimes L
      \ar@{..>}[d]_{\underline\pi^{-1}}\ar@{..>}[r]^{\epsilon^K}
      &\uV(L,\uS(X,Y))\otimes L
      \ar[d]^{\epsilon^L}\\
      \uV(K\otimes L,\uS(X,Y))\otimes K\otimes L
      \ar@{..>}[r]^{\epsilon^{K\otimes L}}
      &\uS(X,Y)
      }}
  \end{equation*}
  Since $\epsilon^L\cdot\underline\phi=\ub\cdot\alpha_{L,X}$ by \eqref{eqn:cp.1.1},
  the diagram \eqref{eqn:cp2.7.21} commutes by \eqref{eqn:cp2.7.7}.
  Now, using the commutativity of \eqref{eqn:cp2.1.28} and \eqref{eqn:cp2.7.21}, one can reduce the commutativity of the diagram $\S(\eqref{eqn:cp2.6.4},Y)$ to the commutativity of the following diagram.
  \begin{equation*}
    \resizebox{\textwidth}{!}{
    \xymatrix{
      \V((K\otimes L)\otimes M,\uS(X,Y))\ar[r]^\pi & \V(K\otimes L,\uV(M, \uS(X,Y)))\ar[r]^\pi &
      \V(K,\uV(L,\uV(M,\uS(X,Y))))\\
      \V(K\otimes (L\otimes M),\uS(X,Y))\ar[u]^{a^*_{K,L,M}} \ar[rr]_\pi && \V(K,\uV(L\otimes M,\uS(X,Y)))\ar[u]_{\underline{\pi}}}}
  \end{equation*}
  But this is a consequence of \eqref{eqn:cp2.2.1}.

  Finally, we define the natural transformation $l$ in \eqref{eqn:cp2.6.2}.
  Given any $X\in\ob\S$, $l_X$ is defined as the unique morphism of $\V$
  that makes the the diagram \eqref{eqn:cp2.7.31} commutes for every $Y\in\ob\S$.
  \begin{equation}
  \label{eqn:cp2.7.31}
    \xymatrix{
      \S(X,Y)\ar[rr]^{l^*_X} \ar[dr]_{\varphi_{X,Y}}
      && \S(I\otimes X,Y)\ar[dl]^{\phi_{I,X,Y}}\\
      &\V(I,\uS(X,Y))}
  \end{equation}
  Because of \eqref{eqn:cp2.1.28}, the commutativity of \eqref{eqn:cp2.6.5} is equivalent to the commutativity of the following diagram of straight arrows.
  \begin{equation*}
    \xymatrix@C=4em{
      \V(K\otimes I,\uS(X,Y)) \ar[r]^\pi &
      \V(K,\uV(I,\uS(X,Y))) & \V(K,\uS((I \otimes X),Y)) \ar[l]_{\underline\phi_{I,X,Y}}\\
      & \V(K,\uS(X,Y)) \ar@{..>}[u]_{i_{\uS(X,Y)}} \ar[ul]^{r^*_K} \ar[ur]_{\uS(l_X,Y)}&}
  \end{equation*}
  The left triangle commutes because of \eqref{eqn:cp2.2.18}.
  Because of \eqref{eqn:cp2.7.8} and the commutativity of \eqref{eqn:cp2.7.31}
  \begin{equation*}
    \varphi_{X,Y}(l_X^{-1})=\alpha_{I,X}
  \end{equation*}
  holds.
  Then using \eqref{eqn:cp2.4.2}, \eqref{eqn:cp2.4.22}, \eqref{eqn:cp.1.1}, \eqref{eqn:cp2.4.9}, one can easily show that
  \begin{equation*}
    \pi^{-1}(\underline\phi_{I,X,Y}\cdot\uS(l_X,Y))=r_{\uS(X,Y)}.
  \end{equation*}
  So, by \eqref{eqn:cp2.2.20}, the right triangle also commutes.

(B) Next, we prove that the correspondence in (1), say $\Phi$, is one-to-one
  by showing
  that the correspondence in (A), say $\Psi$, is the inverse of $\Phi$.

(i) $\Psi\circ\Phi=1$:
  We can decompose \eqref{eqn:cp2.7.4} with $L=I$ into the following diagram.
  \begin{equation*}
    \xymatrix{
      \S((I\otimes K)\otimes X,Y)\ar[r]^{\phi} &
      \V(I\otimes K,\uS(X,Y))\ar[r]^{\pi} &
      \V(I,\uV(K,\uS(X,Y)))\\
      &\uV(K,\uS(X,Y))\ar[ur]^\varpi\ar[u]^{l^*}&\\
      \S(I\otimes (K\otimes X),Y)\ar[uu]^{a^*_{I,K,X}} &
      \uS(K\otimes X,Y)\ar[r]_(0.45){\varphi}\ar[u]_{\phi_{K,X,Y}}
      \ar[l]^{l^*}\ar[luu]_{l^*}
      & \V(I,\uS(K\otimes X,Y))\ar[uu]_{\V(I,\underline{\phi}_{K,X,Y})}}
  \end{equation*}
  Because of the commutativity of \eqref{eqn:cp2.7.6}, the triangle at the lower left corner is a commutative diagram. Since every morphism is an isomorphism, the trapezoid at the lower right corner is also a commutative diagram. So, $(\Psi\circ\Phi)(\phi)=\phi$ by \eqref{eqn:cp2.7.32}. Then by \eqref{eqn:cp2.7.4} and \eqref{eqn:cp2.1.28}, $(\Psi\circ\Phi)(a)=a$.
  $(\Psi\circ\Phi)(l)=l$ holds by \eqref{eqn:cp2.7.12} and \eqref{eqn:cp2.7.31}.

(ii) $\Phi\circ\Psi=1$:
  First, we know by the naturality of $\phi_{\uS(Y,Z),-,Z}$
  \begin{equation}
  \label{eqn:cp2.7.20}
    \phi_{\uS(Y,Z),\uS(X,Y)\otimes X,Z}(\epsilon_Z\cdot (1\otimes\epsilon_Y))=\uS(\epsilon_Y,Z).
  \end{equation}
  Consider the following commutative diagram.
  \begin{equation*}
    \resizebox{1.1\textwidth}{!}{
    \xymatrix@C=4em{
      \uS(Y,Z)\otimes\uS(X,Y)
      \ar[r]^(0.44){\alpha_{\uS(X,Y),X}}\ar[d]_{\uS(\epsilon_Y,Z)}
      &
      \uS(Y,Z)\otimes\uS(X,\uS(X,Y)\otimes X)
      \ar[r]^{\uS(X,\epsilon_Y)}\ar[d]_{\uS(\epsilon_Y,Z)}
      &
      \uS(Y,Z)\otimes\uS(X,Y)
      \ar[dd]^\ub
      \\
      \uS(\uS(X,Y)\otimes X,Z)\otimes\uS(X,Y)
      \ar[r]^(0.44){\alpha_{\uS(X,Y),X}}\ar[d]_{\underline\phi}
      &
      \uS(\uS(X,Y)\otimes X,Z)\otimes\uS(X,\uS(X,Y)\otimes X)
      \ar[d]_\ub
      &
      \\
      \uV(\uS(X,Y),\uS(X,Z))\otimes\uS(X,Y)
      \ar[r]_\epsilon
      &
      \uS(X,Z)
      \ar[r]_=
      &
      \uS(X,Z)
      }}
  \end{equation*}
  Using \eqref{eqn:cp2.7.4} and \eqref{eqn:cp2.7.20}, we know  $\phi_{\uS(Y,Z)\otimes\uS(X,Y),X,Z}(\epsilon_Z\cdot (1\otimes\epsilon_Y)\cdot a)$ is the right upper corner of the above diagram.
  Since $\alpha_{\uS(X,Y),X}=\phi_{\uS(X,Y),X,\uS(X,Y)\otimes X}(1_{\uS(X,Y)\otimes X})$ holds by \eqref{eqn:cp2.7.8},
  the left triangle of the following diagram commutes.
  \begin{equation*}
    \xymatrix@C=5em{
      \uS(X,Y)\otimes X
      \ar[r]^(0.41){\alpha_{\uS(X,Y),X}\otimes 1}\ar[dr]_=
      &
      \uS(X,\uS(X,Y)\otimes X)\otimes X
      \ar[d]^{\epsilon_{\uS(X,Y)\otimes Y}}
      \ar[r]^(0.58){\uS(X,\epsilon_Y)\otimes 1}
      &
      \uS(X,Y)\otimes X
      \ar[d]^{\epsilon_Y}
      \\
      &
      \uS(X,Y)\otimes X
      \ar[r]_{\epsilon_Y}
      &
      Y
      }
  \end{equation*}
  Therefore $\uS(X,\epsilon_Y)\cdot\alpha_{\uS(X,Y),X}=1$ and $(\Phi\circ\Psi)(\ub)=\ub$ holds by  \eqref{eqn:cp2.7.11}.
  $(\Phi\circ\Psi)(\varphi)=\varphi$ holds by \eqref{eqn:cp2.7.12} and \eqref{eqn:cp2.7.31}.
  Theorem \ref{thm:cp2.7.5} implies that $\alpha_{K,X}$ is determined by $\underline\phi_{K,X,Y}$.
  But
  $(\Phi\circ\Psi)(\underline\phi)=\underline\phi$
  holds by \eqref{eqn:cp2.7.4} and \eqref{eqn:cp2.1.28}.
\end{proof}

\section{Closed $\V$-module}

In this final section, we prove Theorem \ref{thm:cp2.1.11}.
In doing so, we are going to use the correspondence in Theorem \ref{thm:cp2.6.9} and its dual without mentioning them explicitly.

First we prove three lemmas needed for the uniqueness in Theorem \ref{thm:cp2.1.11}.

\begin{lemma}
\label{lem:cp2.8.1}
  Let $\V$ be a closed symmetric monoidal category.
  Let $\S$ be a closed $\V$-bimodule.
  Then for every $K,L\in\ob\V$ and $X,Y\in\ob\S$, the following diagram commutes.
  \begin{equation*}
    \xymatrix@C=3.5em{
      \uV(L\otimes K,\uS(X,Y))
      \ar[rr]^c
      \ar[d]_{\underline\pi}
      &&
      \uV(K\otimes L,\uS(X,Y))
      \ar[d]^{\underline\pi}
      \\
      \uV(L,\uV(K,\uS(X,Y)))
      &&
      \uV(K,\uV(L,\uS(X,Y)))
      \\
      \uV(L,\uS(K\otimes X,Y))
      \ar[u]^{\underline\phi_{K,X,Y}}
      &
      \uS(K\otimes X,L\pitchfork Y)
      \ar[r]^(0.48){\underline\phi_{K,X,L\pitchfork Y}}
      \ar[l]_(0.48){\underline\psi_{L,Y,K\otimes X}}
      &
      \uV(K,\uS(X,L\pitchfork Y))
      \ar[u]_{\underline\psi_{L,Y,X}}}
  \end{equation*}
\end{lemma}
\begin{proof}
  In the above diagram, we have two morphisms from $\uS(K\otimes X,L\pitchfork Y)$ to $\uV(K\otimes L,\uS(X,Y))$. We want to show that they coincide after applying $\pi^{-1}_{-,K\otimes L,-}$ to them.
  Consider the following diagram.
  \begin{equation*}
    \resizebox{1.1\textwidth}{!}{
    \xymatrix@C=0em@R=2em{
      \uS(K\otimes X,L\pitchfork Y)\otimes(K\otimes L)
      \ar[r]^(0.45){\alpha_{K,X}}
      \ar[d]_{\underline\phi_{K,X,L\pitchfork Y}}
      &
      (\uS(K\otimes X,L\pitchfork L)\otimes\uS(X,K\otimes X))\otimes L
      \ar[d]^\ub
      &
      \\
      \uV(K,\uS(X,L\pitchfork Y))\otimes(K\otimes L)
      \ar[r]^\epsilon
      \ar[dd]_{\underline\psi_{L,Y,X}}
      &
      \uS(X,L\pitchfork Y)\otimes L
      \ar[r]^(0.45){\beta_{L,Y}}
      \ar[dd]^{\underline\psi_{L,Y,X}}
      &
      \uS(X,L\pitchfork Y)\otimes \uS(L\pitchfork Y,Y)
      \ar[d]^c
      \\
      &&
      \uS(L\pitchfork Y,Y)\otimes \uS(X,L\pitchfork Y)
      \ar[d]^\ub
      \\
      \uV(K,\uV(L,\uS(X,Y)))\otimes(K\otimes L)
      \ar[r]^\epsilon
      \ar[d]_{-\otimes L}
      &
      \uV(L,\uS(X,Y))\otimes L
      \ar[r]^\epsilon
      &
      \uS(X,Y)
      \\
      \uV(K\otimes L,\uV(L,\uS(X,Y))\otimes L)\otimes (K\otimes L)
      \ar[rr]^\epsilon
      \ar[ur]^\epsilon
      &&
      \uV(K\otimes L,\uS(X,Y))\otimes (K\otimes L)
      \ar[u]_\epsilon}}
  \end{equation*}
  By \eqref{eqn:cp2.7.16} and the commutativity of the diagrams \eqref{eqn:cp.1.1} and \eqref{eqn:cp2.1.25}, the above diagram commutes.
  Then, using Lemma \ref{lem:cp2.7.13}, we obtain
  \begin{equation*}
    \pi^{-1}(\underline\pi^{-1} \cdot\underline\psi_{L,Y,X} \cdot\underline\phi_{K,X,L\pitchfork Y})=\ub\cdot (1\otimes\ub)\cdot c \cdot a^{-1}\cdot (1\otimes(\alpha_{K,X}\otimes\beta_{L,Y}))
  \end{equation*}
  From a similar diagram chasing, we obtain
  \begin{equation*}
    \pi^{-1}(c\cdot \underline\pi^{-1} \cdot\underline\phi_{K,X,Y} \cdot\underline\psi_{L,Y,K\otimes X})=\ub\cdot (\ub\otimes 1)\cdot c\otimes 1 \cdot a^{-1}\cdot 1\otimes c\cdot (1\otimes(\alpha_{K,X}\otimes\beta_{L,Y}))
  \end{equation*}
  Then from the commutativity of the diagram \eqref{eqn:elno.4.12},
  \begin{equation*}
    \pi^{-1}(\underline\pi^{-1} \cdot\underline\psi_{L,Y,X} \cdot\underline\phi_{K,X,L\pitchfork Y})=\pi^{-1}(c\cdot \underline\pi^{-1} \cdot\underline\phi_{K,X,Y} \cdot\underline\psi_{L,Y,K\otimes X}).
  \end{equation*}
  holds.
\end{proof}

\begin{lemma}
\label{lem:cp2.8.2}
  Let $\V$ be a closed symmetric monoidal category.
  Let $\S$ be a closed $\V$-bimodule.
  Then, for every $K,L\in\ob\V$ and $X,Y\in\ob\S$, the following diagram commutes.
  \begin{equation}
  \label{eqn:cp2.8.1}
    \xymatrix{
      \S((L\otimes K)\otimes Y,X)
      \ar[r]^c
      &
      \S((K\otimes L)\otimes Y,X)
      \ar[r]^{\psi^{-1}\cdot\phi}
      &
      \S(Y,(K\otimes L)\pitchfork X)
      \\
      \S(L\otimes (K\otimes Y),X)
      \ar[r]_{\psi^{-1}\cdot\phi}
      \ar[u]^{a^*}
      &
      \S(K\otimes Y,L\pitchfork X)
      \ar[r]_{\psi^{-1}\cdot\phi}
      &
      \S(Y,K\pitchfork (L\pitchfork X))
      \ar[u]_{a^{\op}}
      }
  \end{equation}
\end{lemma}
\begin{proof}
  Consider the following commutative diagram.
  \begin{equation*}
    \xymatrix@C=-3em@R=2em{
      \S((L\otimes K)\otimes Y,X)
      \ar[dr]^{{\phi}}
      \ar[rr]^c
      &&
      \S((K\otimes L)\otimes Y,X)
      \ar[dr]^{{\phi}}
      &&
      \S(Y,(K\otimes L)\pitchfork X)
      \ar[dl]_{{\psi}}
      \\
      &\V(L\otimes K,\uS(Y,X))
      \ar[rr]^c
      \ar[d]_{{\pi}}
      &&
      \V(K\otimes L,\uS(Y,X))
      \ar[d]^{{\pi}}
      &
      \\
      &\V(L,\uV(K,\uS(Y,X)))
      &&
      \V(K,\uV(L,\uS(Y,X)))
      &
      \\
      &
      \V(L,\uS(K\otimes Y,X))
      \ar[u]^{\underline{\phi}}
      &&
      \V(K,\uS(Y,L\pitchfork X))
      \ar[u]_{\underline{\psi}}
      &
      \\
      \S(L\otimes (K\otimes Y),X)
      \ar[ur]_{{\phi}}
      \ar[uuuu]^{a^*}
      &&
      \S(K\otimes Y,L\pitchfork X)
      \ar[ru]_{{\phi}}
      \ar[ul]^{{\psi}}
      &&
      \S(Y,K\pitchfork (L\pitchfork X))
      \ar[ul]^{{\psi}}
      \ar[uuuu]_{a^{\op}}
      }
  \end{equation*}
  The left and the right trapezoids commute by \eqref{eqn:cp2.1.28} and its dual.
  The middle commutes by Lemma \ref{lem:cp2.8.1}.
  Hence the result holds.
\end{proof}

\begin{lemma}
\label{lem:cp2.8.3}
  Let $\V$ be a closed symmetric monoidal category.
  Let $\S$ be a closed $\V$-bimodule.
  Then, for every $X,Y\in\ob\S$, the following diagram commutes.
  \begin{equation}
  \label{eqn:cp2.8.2}
    \xymatrix{
      \S(I\otimes Y,X)
      \ar[r]^\phi
      &
      \V(I,\uS(Y,X))
      &
      \S(Y,I\pitchfork X)
      \ar[l]_\psi
      \\
      &
      \S(Y,X)
      \ar[ul]^{l_Y^*}
      \ar[ur]_{(l_X^{\op})_*}
      &
      }
  \end{equation}
\end{lemma}
\begin{proof}
  $\phi\cdot l_Y^*=\varphi_{Y,X}=\varphi^{\op}_{X,Y}=\psi\cdot (l^{\op}_X)_*$ by \eqref{eqn:cp2.1.34}.
\end{proof}

\begin{proof}[Proof of Theorem \ref{thm:cp2.1.11}]
  The uniqueness of $a^{\op}$ and $l^{\op}$ follows from Lemma \ref{lem:cp2.8.2} and Lemma \ref{lem:cp2.8.3} respectively.

  We define a morphism $a^{\op}:K\pitchfork(L\pitchfork X)\rightarrow (K\otimes L)\pitchfork X$ of $\S$ natural in $K,L,X$ by Yoneda lemma and the commutativity of the diagram \eqref{eqn:cp2.8.1}.
  By the definition of $a^{\op}$,  the commutativity of \eqref{eqn:cp2.1.42} is reduced to the commutativity of
  the following diagram
  \begin{equation*}
    \resizebox{\textwidth}{!}{
    \xymatrix@R=2em{
      \S(((K\otimes L)\otimes M)\otimes Y, X)
      &
      \S((M\otimes (K\otimes L))\otimes Y, X)
      \ar[l]_c
      &
      \S(M\otimes((K\otimes L)\otimes Y),X)
      \ar[l]_a
      \\
      \S((K\otimes (L\otimes M))\otimes Y, X)
      \ar[u]^a
      &
      &
      \S(M\otimes((L\otimes K)\otimes Y),X)
      \ar[u]^c
      \\
      \S(((L\otimes M)\otimes K)\otimes Y, X)
      \ar[u]^c
      &
      &
      \\
      \S((L\otimes M)\otimes (K\otimes Y), X)
      \ar[u]^a
      &
      \S((M\otimes L)\otimes (K\otimes Y), X)
      \ar[l]_c
      &
      \S(M\otimes(L\otimes (K\otimes Y)),X),
      \ar[uu]_a
      \ar[l]_a
      }}
  \end{equation*}
  hence, to the commutativity of the following diagram of straight arrows.
  \begin{equation*}
    \xymatrix@R=2em{
      ((K\otimes L)\otimes M)\otimes Y
      \ar[r]^c
      \ar[d]_{a}
      &
      (M\otimes (K\otimes L))\otimes Y
      \ar[r]^{a}
      \ar@{..>}[d]^c
      &
      M\otimes((K\otimes L)\otimes Y)
      \ar[d]^c
      \\
      (K\otimes (L\otimes M))\otimes Y
      \ar[d]_c
      &
      (M\otimes (L\otimes K))\otimes Y
      \ar@{..>}[r]^{a}
      &
      M\otimes((L\otimes K)\otimes Y)
      \ar[dd]^{a}
      \\
      ((L\otimes M)\otimes K)\otimes Y
      \ar[d]_{a}
      \ar@{..>}[r]^c
      &
      ((M\otimes L)\otimes K)\otimes Y
      \ar@{..>}[u]_{a}
      \ar@{..>}[d]^{a}
      &
      \\
      (L\otimes M)\otimes (K\otimes Y)
      \ar[r]^c
      &
      (M\otimes L)\otimes (K\otimes Y)
      \ar[r]^{a}
      &
      M\otimes(L\otimes (K\otimes Y))
      }
  \end{equation*}
  The upper left square commutes by \eqref{eqn:elno.4.12}.
  The lower right square commutes by \eqref{eqn:cp2.6.4}.
  Therefore, the boundary is a commutative diagram.

  Next, we define a morphism $l^{\op}:X\rightarrow I\pitchfork X$ of $\S$  natural in $X$ by Yoneda lemma and the commutativity of the diagram \eqref{eqn:cp2.8.2}.
  By the definitions of $l^{\op}$ and $a^{\op}$, the commutativity of \eqref{eqn:cp2.1.43} is reduced to the commutativity of
  the following diagram of straight arrows.
  \begin{equation*}
    \xymatrix{
      \S((K\otimes I)\otimes Y,X)
      \ar[r]^c
      \ar[d]_r
      &\S((I\otimes K)\otimes Y,X)
      \ar[d]^a
      \ar@{..>}[dl]_l
      \\
      \S(K\otimes Y,X)
      &
      \S(I\otimes(K\otimes Y),X)
      \ar[l]_l}
  \end{equation*}
  But the upper triangle commutes by \eqref{eqn:elno.4.13} and
  the lower triangle commutes by the commutativity of \eqref{eqn:cp2.7.6}.
\end{proof}

\bibliographystyle{amsalpha}
\providecommand{\bysame}{\leavevmode\hbox to3em{\hrulefill}\thinspace}
\providecommand{\MR}{\relax\ifhmode\unskip\space\fi MR }
\providecommand{\MRhref}[2]{%
  \href{http://www.ams.org/mathscinet-getitem?mr=#1}{#2}
}
\providecommand{\href}[2]{#2}

\end{document}